\documentclass{amsart}
\usepackage{enumerate}
\usepackage{amsmath}
\usepackage{amsfonts}
\usepackage{mathrsfs}
\usepackage{enumerate}
\usepackage{galois}
\usepackage[toc,page]{appendix}
\usepackage{xcolor}
\usepackage[colorlinks, linkcolor = blue, anchorcolor = blue, citecolor = blue]{hyperref}
\usepackage{eso-pic}
\usepackage{amssymb}
\usepackage{soul}
\usepackage{graphicx}
\usepackage{mathtools}

\newtheorem{theorem}{Theorem}[section]

\newtheorem{alphatheorem}{Theorem}[section]

\newtheorem{proposition}[theorem]{Proposition}
\newtheorem{claim}{Claim}

\newtheorem{lemma}[theorem]{Lemma}

\theoremstyle{definition}
\newtheorem{definition}[theorem]{Definition}
\newtheorem{remark}[theorem]{Remark}

\DeclareMathOperator{\an}{An}

\DeclareMathOperator{\Sing}{Sing}
\DeclareMathOperator{\Reg}{Reg}
\DeclareMathOperator{\Per}{Per}
\DeclareMathOperator{\Vol}{Vol}
\DeclareMathOperator{\dist}{dist}
\DeclareMathOperator{\spt}{spt}
\DeclareMathOperator{\Id}{Id}

\newcommand\res{\mathop{\hbox{\vrule height 7pt width .5pt depth 0pt
      \vrule height .5pt width 6pt depth 0pt}}\nolimits}

\numberwithin{equation}{section}

\usepackage[style=alphabetic, backend=biber, sorting=nyt, url=false ]{biblatex}
\title[Anisotropic minimal hypersurfaces]{Existence and regularity of min-max anisotropic minimal hypersurfaces}
\author{Guido De Philippis}
\address{Courant Institute of Mathematical Sciences, New York University, 251 Mercer St., New York, NY 10012, USA}
\email{guido@cims.nyu.edu}
\author{Antonio De Rosa}
\address{Department of Decision Sciences and BIDSA, Bocconi University, Milano, Italy}
\email{antonio.derosa@unibocconi.it}
\author{Yangyang Li}
\address{The University of Chicago, Department of Mathematics, Eckhart Hall,
5734 S University Ave, Chicago, IL 60637, USA}
\email{yangyangli@uchicago.edu}

\begin{document}

\begin{abstract}
  In any closed smooth Riemannian manifold of dimension at least three, we use the min-max construction to find anisotropic minimal hyper-surfaces with respect to elliptic integrands, with a singular set of codimension~$2$ vanishing Hausdorff measure. In particular, in a closed $3$-manifold, we obtain a smooth anisotropic minimal surface.

  The critical step is to obtain a uniform upper bound for density ratios in the anisotropic min-max construction. This confirms a conjecture by Allard [Invent. Math., 1983].

\end{abstract}

\maketitle

\section{Introduction}

\subsection*{Background and main results}

Minimal surfaces are critical points of the area functional. The existence of minimal surfaces in every Riemannian manifold has been a central research theme in Geometric Measure Theory over the past 50 years. When the topology of the ambient manifold is sufficiently rich, existence can be established via suitable minimization problems. However, if the topology is too simple, for example in a sphere,  minimization problems may only have trivial solutions, necessitating different techniques.

In~\cite{Almgren1962,Almgren1965} Almgren started a program to develop a geometric version of the Calculus of Variation in the large to show existence of (not necessarily minimizing) critical points for the area functional. In the case of codimension one surfaces, the program was completed by Pitts in~\cite{Pitts1981}, where he proved that every \(n+1\) Riemannian  manifold contains a non-trivial \(n\) dimensional minimal surface, as long as \(n+1\le 6\). Schoen and Simon~\cite{SchoenSimon} extended Pitts's technique to show that, in general, every \(n+1\) dimensional manifold admits a open \(n\)-dimensional minimal surface \(\Sigma\), whose singular set \(\overline{\Sigma}\setminus\Sigma\) is of  dimension at most \(n-7\), in particular matching the optimal regularity for solutions of the minimization problem. In recent years, Almgren-Pitts min-max techniques have received a renewed attention and have been a key tool in solving a series of long standing problems in Geometric Analysis~\cite{Marques-Neves2014,Marques-Neves2016, AMN2016, Song2023,Li2023}.

In this paper we are interested in anisotropic surface tensions of the form
\[
  \mathbf{\Phi}(\Sigma) = \int_\Sigma \Phi(x, \nu_\Sigma) d\mathcal{H}^{n}_g\,,
\]
where $\Phi$ is an even elliptic integrand defined on the unit tangent bundle of $M$,  $\nu_\Sigma$ is the unit normal vector of $\Sigma$ and \(\mathcal{H}^{n}_g\) is the \(n\)-dimensional volume measure associated with the Riemannian metric \(g\).  Extending the existence and regularity theory of minimal surfaces to critical points of anisotropic energies has been a central theme of research in Geometric Measure Theory, starting from their introduction in~\cite{Almgren68}. In spite of these effort, little progress has been done over the past 60 years. Indeed while, at least in codimension one, the existence and regularity  theory for minimizers of anisotropic energies  largely parallels the one of  the area functional~\cite{ottimaleASS}, a satisfactory theory is completely missing for what concerns critical points.

The main reason for this difference is the absence of any general way to deduce local area bounds for  critical points of anisotropic energies. These local bounds are indeed   a crucial ingredient in performing the local blow-up analysis which is at heart of the regularity theory for stationary varifolds,~\cite{Allard1972}. In the case of the area functional, these bounds are ensured by the validity of \emph{monotonicity formulas} which are a key tool in the study of minimal surfaces. Allard has however showed in~\cite{Allard1974} that the validity of a monotonicity type  formula essentially characterizes the area integrand. The analysis of critical points for anisotropic energies  thus proves to be extremely more challenging and requires the introduction of new ideas.

In~\cite{DePhilippis_DeRosa2024}, the first two named authors have been able to extend the Almgren-Pitts theory (as modified by Colding-De Lellis,~\cite{CDL2002}) to prove that in any \(3\)-dimensional manifold, one can always find a critical surface \(\Sigma\) for \(\Phi\) which is smooth, with the possible exception of one singular point. The dimensional restriction plays a crucial role in~\cite{DePhilippis_DeRosa2024}. Indeed, as mentioned above, in any regularity argument a key step is to show that critical points constructed via min-max enjoy local area bounds. To ensure these bounds, one uses that the min-max construction essentially ensures stability on the complement of the point, which in turn implies an \(L^2\) control on the second fundamental form. The \(L^2\) norm of the second fundamental form is a critical quantity for \(2\) dimensional surfaces, and one can use it to obtain a local control on the area.

In this paper we remove the dimensional restriction of~\cite{DePhilippis_DeRosa2024} and we prove the existence of critical points  for anisotropic energies in any dimension, in particular providing a positive answer to a conjecture posed by Allard in 1983,~\cite[Page 288]{Allard}:

\begin{alphatheorem}\label{thm:main1}
  Given any smooth closed Riemannian manifold $(M^{n+1}, g)$ with $n + 1\geq 3$ and any smooth elliptic integrand $\Phi$, there exists a smooth embedded anisotropic minimal hypersurface $\Sigma^n$ such that $\Sing(\Sigma) \equiv \overline{\Sigma}\setminus \Sigma$ has $\mathcal{H}^{n-2}(\Sing(\Sigma)) = 0$. In particular, when $n + 1 = 3$, there exists a smooth anisotropic minimal surface.
\end{alphatheorem}

Note that  the regularity established for  $\Sigma$ in Theorem~\ref{thm:main1}  matches the one  for minimizers of the anisotropic Plateau problem for currents in codimension one ~\cite{ottimaleASS}, and, given the counterexample in~\cite{Morgan1990}, it is essentially optimal.

Theorem~\ref{thm:main1} is a consequence of  (a stronger version of) our  second main result, which concerns the existence of hypersurfaces with non-zero constant anisotropic mean curvature. We refer to the next section for the notion of almost embedded surface.

\begin{alphatheorem}\label{thm:main2}
  Given any smooth closed Riemannian manifold $(M^{n+1}, g)$ with $n + 1\geq 3$, any smooth elliptic integrand $\Phi$, and $c\in \mathbb{R}\setminus \{0\}$, there exists a smooth almost embedded hypersurface $\Sigma^n$ with constant anisotropic mean curvature $c$ with respect to $\Phi$, such that $\operatorname{Sing}(\Sigma) \equiv \overline{\Sigma}\setminus \Sigma$ has $\mathcal{H}^{n-2}(\operatorname{Sing}(\Sigma)) = 0$.
\end{alphatheorem}

We conclude the first part of this introduction by noticing that, as a corollary of the above results, we can prove that in every Finsler manifold there are minimal surfaces for the  Holmes–Thompson volume; see~\cite{AT04}.

\begin{alphatheorem}\label{t:existenceFinsler}
  Let $M$ be an $(n+1)$-dimensional smooth Finsler manifold  with $n + 1 \geq 3$, such that the norms on all tangent spaces $T_xM$ are uniformly convex. Then there is a nontrivial minimal hypersurface $\Sigma^n\subset M$ with respect to the Holmes–Thompson volume, without boundary, which is smooth away from a singular set $\Sing(\Sigma)$ with  $\mathcal{H}^{n-2}(\Sing(\Sigma)) = 0$. More generally, for every \(c>0\), there exists an almost embedded surface with mean curvature equal to \(c\).
\end{alphatheorem}

\subsection*{Strategy of the proof and structure of the paper}

The general strategy for proving Theorem~\ref{thm:main1} and~\ref{thm:main2} is based on the Almgren-Pitts min-max construction which we now briefly summarize. Relying on the Almgren isomorphism~\cite{Almgren1962}, one shows that there always exist nontrivial \emph{one-parameter} sweepouts of \(M\) by a one parameter family of codimension one cycles. This sets up a mountain pass geometry for the area functional, from which, together with a simple pull tight procedure, one easily shows the existence of a stationary varifold which realizes the min-max value. However, stationary varifolds are known to be smooth only on the complement of a possibly big set. To prove the existence of a regular minimal surface, one has to further rely on the variational construction of the min-max critical point. Indeed, the construction roughly ensures that its Morse index is at most one, which essentially implies that it is stable in the neighborhood of every point in its support, with the possible exception of one. The key idea of Pitts was to upgrade this stability to an almost minimizing property,~\cite{Pitts1981}, and through various delicate estimates, ultimately prove its regularity.

For anisotropic energies, the same idea can be employed to prove the existence of a \(\Phi\)-stationary varifold that is \(\Phi\) almost  minimizing. However, in this general setting, no regularity theory for such surfaces is known. Once again, the obstruction is the absence of a local area control, which prevents the use of crucial blow-up arguments.

The new key idea here is to use the min-max construction to show that we can indeed produce a \(\Phi\)-almost  minimizing varifold \(V\) that also satisfies the desired area control. In order to obtain these local area bounds, we rely on the notion of \emph{nested} sweepouts introduced in~\cite{Chambers-Liokumovich2020} and~\cite{Chodosh-Liokumovich-Spolaor2022}. This nested property, roughly speaking, allows us to prove that the resulting critical point is \emph{one-sided} minimizing, which, by comparison with small balls, would ensure a local energy bound (though the actual argument is however more involved).

Once a local area bound is achieved, one can attempt to run Pitts' argument to show the regularity of the constructed varifold. The theory is based on the notion of \emph{replacement}. More precisely one can prove that for any small annuli \(A\), there exists a globally stationary varifold \(V^*\), which coincides with \(V\) outside \(A\) and that is a smooth and stable \(\Phi\)-minimal surface in \(A\).
The goal is to show that they globally coincide. To achieve this, one performs a local blow up at the (exterior) boundary of \(A\) (though again the actual argument is more involved) and  attempts to leverage the fact that the limiting surface is indeed globally stationary to show that the tangent planes of \(V\) and \(V^*\) are the same across this boundary. In the area case, one relies on the monotonicity formula, which implies that the blow up is indeed a cone; this, together with the global stationarity, leads to the desired conclusion. In our case, however, we must instead rely on a careful barrier construction, inspired by the one in~\cite{DPM1} (see also~\cite{Hardt77} for similar arguments). In order to carry out these arguments, we need to ensure that the constructed varifold has multiplicity one. This is why, as in~\cite{DePhilippis_DeRosa2024}, we first construct varifold with \(\Phi\) mean curvature equal to a positive constant \(c>0\). The maximum principle indeed ensures that for such surfaces, the multiplicity is \(1\), except on a set of dimension \((n-1)\). We then pass to the limit as \(c\to 0\) to complete the proof.

A final remark concerns the proof of the smoothness of the replacement \(V^*\) in \(A\). By construction, it is obtained as the limit of globally stable (in \(A\)) surfaces, which are  local minimizers at decreasing scales. These minimizers are smooth by the classical theory developed in~\cite{ottimaleASS}, but one has to prove global a priori estimates to pass to the limit. For the area functional, these estimates have been proved by Schoen-Simon in~\cite{SchoenSimon}, also allowing for a small singular set. For anisotropic energies, Allard has extended these estimates in~\cite{Allard}, \emph{provided one knows a priori a local area bound} (to be precise, Allard does not allow for a singular set, but this requires only minor changes). We conclude by noticing that a new simplified proof of~\cite{SchoenSimon} has been recently been given by Bellettini in~\cite{Bel2023}. In a forthcoming work, we plan to extend that result to the setting of anisotropic energies.

\medskip

The paper is structured as follows: in Section~\ref{sec:preliminaries} we collect few standard fact in GMT together with the main properties of anisotropic energies we will need in the sequel. In Section~\ref{sec:sweepouts} we construct a volume parameterized optimal sweepout with  bounded area, in Section~\ref{sec: Alomost minimizing varifolds and unifrom mass ratio} we prove the existence of an almost minimizing varifold with bounded mass ratio, while in Section~\ref{sec:Regularity of min-max minimal hypersurfaces} we prove its regularity. Eventually in Section~\ref{sec:mainproofs} we combine these results to prove the main theorems. Appendix~\ref{sec:compactness} contains the main compactness we exploit in the paper, while Appendix~\ref{sec:geometric} contains the proof of a simple geometric lemma.

\subsection*{Acknowledgments}
Antonio De Rosa was funded by the European Union: the European Research Council (ERC), through StG ``ANGEVA'', project number: 101076411. Guido De Philippis is supported by the NSF grant DMS 2055686 and by the Simons Foundation. Yangyang Li was partially supported by the AMS-Simons travel grant.

Views and opinions expressed are however those of the authors only and do not necessarily reflect those of the European Union or the European Research Council. Neither the European Union nor the granting authority can be held responsible for them.

\section{Preliminaries}\label{sec:preliminaries}

\subsection{Terminology}
Throughout the paper, we fix a smooth closed (i.e.\ compact and without boundary) \((n+1)\)-dimensional Riemannian manifold $(M^{n+1}, g)$ with \((n+1)\ge 3\). Here by ``smooth'', we mean that both the manifold and the metric are \(C^\infty\). While this is definitely not the most optimal assumption, a careful analysis of the arguments below will show that \(C^3\) regularity would suffice. However, since this is not a salient point, we will not pursue this refinement further. We also note that the fixed background Riemannian metric plays essentially no role in the following, since it can be absorbed into the integrand. It is however useful to fix a background volume form and to identify \(n\) dimensional planes with their normals.

We assume the reader to be familiar with the standard notions in Geometric Measure Theory,~\cite{Simon1984}, and we will adapt the following notations and conventions.
\begin{itemize}
\item $UT(M)$:  the unit tangent bundle of $M$, namely,
  \[
    UT(M) := \{(x, v) \in TM: \|v\|_g = 1\}\,;
  \]
\item $G_n(M)$:  the unoriented hyperplane bundle of $M^{n+1}$, namely,
  \[
    G_n(M) := \{(x, T): x \in M, T \text{ is an $n$-dimensional linear subspace of } T_x M\}\,.
  \]
 By means of the background metric, we can identify \(G_n(M)\) with \( UT(M)/\sim\), where we have set the equivalence relation \((x,v) \sim (x, -v)\).
\item $\mathrm{inj}(M)$: the injective radius of $(M, g)$.
\item $\mathcal{X}(M)$: the space of smooth vector fields on $M$.
\item $\mathcal{C}(M)$: the space of Caccioppoli sets (sets of finite perimeter) in $M$, endowed with the topology induced by the \(L^1\) distance of the characteristic functions. In particular, two sets in $\mathcal{C}(M)$ are considered identical if they differ by an $\mathcal{H}^{n+1}$ measure zero set. We denote the reduced boundary of a Caccioppoli set $\Omega\in \mathcal{C}(M)$ by $\partial^* \Omega$.
\item $\mathcal{Z}^0_n(M; \mathbb{Z}_2)$: the space of modulo two $n$-cycles on $M$ in the connected component containing $0$, endowed with the flat metric $\mathcal{F}$ topology;
\item $\mathcal{V}(M) = \mathcal{V}_n(M^{n+1})$: the space of $n$-varifolds on $M$, namely, the space of non-negative Radon measure on $G_n(M)$;
\item $\mathcal{RV}(M) = \mathcal{RV}_n(M^{n+1})$: the space of $n$-rectifiable varifolds on $M$;
\item $\mathcal{IV}(M) = \mathcal{IV}_n(M^{n+1})$: the space of integral $n$-rectifiable varifolds on $M$;
\item $|K|:=\mathcal{H}^n\res K \otimes \delta_{T_xK}$: the integral varifold associated to an $n$-rectifiable set $K$, or to a modulo two $n$-cycle $K$;
\item $\|V\|, \|T\|$: the associated Radon measure on $M$ of $V \in \mathcal{V}(M)$ and $T \in \mathcal{Z}^0_n(M; \mathbb{Z}_2)$;
\item  $T_x^r : z\in B_1(0)\to \exp_x(rz) \in B_r(x)$, where $\exp_x$ denotes the exponential map at the point $x$, for every $x\in M$ and $r<\mathrm{inj}(M)$. We set  $\eta_{x,r}(y):=(T_x^r)^{-1}(y)$;
\item $TV(x,V)$: the set of all the sub-sequential limits, as $r \to 0$, of $(\eta_{x,r})_\#V \in \mathcal V(\mathbb R^{n+1})$, where $x\in M$ and $V \in \mathcal V(M)$;
\item \(TV(x,\Omega)\) the set of all the sub-sequential limits, as  $r \to 0$, of $\eta_{x,r}(\Omega) \in \mathcal{C}(M)$, where $x\in M$ and $\Omega \in \mathcal{C}(M)$.
\end{itemize}

\subsection{Anisotropic energies}

In the sequel, an anisotropic integrand will refer to a smooth function $\Phi: G_n(M) \to \mathbb{R}^+$. Note that according to the identification above, we can consider it as a function \(\Phi: U(TM)\to \mathbb{R}^+\), which is an even function in the second variable. When no confusion arises, we will often switch between these points of view without explicitly  acknowledging it. It will also be useful to extend \(\Phi\) one homogeneously in the second variable as
\[
\Phi(x, v)=|v|\Phi\Big(x, \frac{v}{|v|}\Big).
\]
In the sequel, we will \emph{always assume} that this extension has been made whenever we consider derivatives of \(\Phi\).

We will say that the integrand is \emph{elliptic} if the map \(v \mapsto \Phi(x, v)\) is convex, and we will say the integrand is \emph{uniformly elliptic} if the same map is uniformly convex in all directions except the radial one. In order to quantify these properties, we fix a parameter \(\lambda\) and make the following assumptions:
\begin{equation}\label{limitato}\tag{$H_\lambda$}
      \frac{1}{\lambda}\leq \Phi \vert_{U(TM)} \leq \lambda \qquad   \|\Phi\big\|_{C^3(UT(M))} \leq \lambda
       \qquad
      D^2_v  \Phi\vert_{UT(M)}\geq \frac{\mathrm{Id}_{v^\perp}}{\lambda}.
\end{equation}
For a varifold $V \in \mathcal{V}(M)$, we define its $\Phi$-anisotropic energy as
\[
  \mathbf{\Phi}(V) =  \int_{G_n (M)} \Phi(x, T) dV(x, T)\,.
\]
Note that by the identification of the \(n\)-dimensional Grassmanian with the unit sphere, we can equally thing an \(n\)-dimensional varifold as a measure on \(UT(M)\). In that case we will write:
\[
  \mathbf{\Phi}(V) =  \int_{UT (M)} \Phi(x, \nu) dV(x, \nu)\,.
\]
For a modulo two $n$-cycle $S \in \mathcal{Z}^0_n(M; \mathbb{Z}_2)$ or an $n$-rectifiable set $S$, we will often use the lighter notation
\[
  \mathbf{\Phi}(S) := \mathbf{\Phi}(|S|) = \int_{S} \Phi(x, T_x S) d\mathcal{H}_{g}^n(x)\,.
\]
When \(S=\partial^{*}\Omega\), where \(\Omega\) is a set of finite perimeter, we will also make the following abuse of notation, when no confusion arises:
\[
  \mathbf{\Phi}(\Omega) := \mathbf{\Phi}(|\partial^{*}\Omega|) = \int_{\partial ^{*} \Omega} \Phi(x, \nu_\Omega(x)) d\mathcal{H}_{g}^n(x)\,.
\]
When we will need to localize the integration to a Borel subset \(U\subset M\), we will use the notation:
\begin{equation}\label{e:loc}
  \mathbf{\Phi}(V;U) :=  \int_{G_n (U)} \Phi(x, T) dV(x, T)\,,
\end{equation}
where \(G_n (U)\) is the restriction of the Grassmannian bundle to \(U\). Note that we are not making any assumption on \(U\), although most of the time, \(U\) will be either an open set or the closure of an open set. We use a similar notation to localize the $\Phi$-anisotropic energy of a cycle or of a set of finite perimeter.

\begin{remark}\label{rmk:Phi-norm}
  On the vector space $\mathcal{Z}^0_n(M; \mathbb{Z}_2)$, $\mathbf{\Phi}$ is in fact a norm, and any two integrands satisfying~\eqref{limitato} induce the same topology, which is finer than the flat topology,~\cite{Almgren1965,Pitts1981}.

\end{remark}

The \emph{first variation} of the $\Phi$-anisotropic energy is defined by
\[
  \delta_\mathbf{\Phi} V(X) := \frac{d}{dt}\Bigl |_{t = 0} \mathbf{\Phi}\left((\varphi_t)_\# V\right)
\]
where $V \in \mathcal{V}(M)$, $X \in \mathcal{X}(M)$ and $\varphi_t$ is defined by $\frac{d\varphi_t}{dt} = X$ and $\varphi_0 = \operatorname{Id}$.  Referring to~\cite{DPDRGrect} for the general expression, here we record that when \(M=\mathbb R^{n+1}\) and \(\Phi(x,\nu)=\Phi(\nu)\) does not depends on \(x\), we have the following formula:
\begin{equation}
\label{e:firstvariation}
 \delta_\mathbf{\Phi} V(X)=\int \Phi (\nu)\operatorname{div} X-\langle D\Phi(\nu), DX^T\nu \rangle dV(x,\nu)
\end{equation}
where \(DX^T\) is the transpose of \(DX\). In the general case where  \(\Phi\) depends on \(x\) as well, one need to add a term which depends on \(D_x\Phi\), \cite{DDH}. For a general varifold, the following estimates always holds true:
\begin{equation}
    \label{e:firstvaritationroughbound}
    |\delta_\mathbf{\Phi} V(X)|\le C(\lambda)\|V\|\bigl(\spt X)(\|DX\|_{\infty}+\|X\|_{\infty}\bigr)
\end{equation}
A varifold $V \in \mathcal{V}(M)$ is said to have locally bounded $\Phi$-anisotropic first variation if there exists $C > 0$ such that for all $X \in \mathcal{X}(M)$,
\[
  |\delta_{\mathbf{\Phi}} V(X)| \leq C \|X\|_\infty.
\]
Equivalently, an $n$-varifold $V \in \mathcal{V}(M)$ has locally bounded $\Phi$-anisotropic first variation if \(\delta_{\mathbf{\Phi}} V\) is a Radon measure. By the Radon-Nikodym theorem, there exists an $L^1(\|V\|)$ vector $H^V_\Phi$ and a vector-valued measure \(\eta_V\), \(\eta_V\perp \|V\|\), such that for all $X \in \mathcal{X}(M)$,
\[
  \delta_\mathbf{\Phi} V(X) = - \int_M X \cdot H^V_\Phi d\|V\|+\int_M X \cdot d\eta_V.
\]
In this case, $H^V_\Phi$ is called the \emph{generalized $\Phi$-anisotropic mean curvature} of $V$.  In case \(\eta_V=0\) and \( |H^V_\Phi|\le C\),  we say that $V$ has \emph{$\Phi$-anisotropic first variation bounded by} $C$.

We will say that  $V$ is \emph{$\Phi$-stationary} if $\delta_\mathbf{\Phi} V \equiv 0$. In the case that a \(\Phi\)-stationary varifold is associated with a smooth, properly embedded hypersurface \(\Sigma\), we will say that \(\Sigma\) is \(\Phi\)-minimal. We refer to~\cite{SurveyDeRosa} for a survey on the theory of anisotropic minimal surfaces.

\subsection{Anisotropic CMC hypersurfaces}\label{sec:cmc}
In the sequel, we will work with anisotropic constant mean curvature (CMC) hypersurfaces. We will use the tools from~\cite[Section 2.4]{DePhilippis_DeRosa2024}, which is the anisotropic counterpart of~\cite[Section 2]{Zhou-zhu2020}. For any non-negative constant $c \in [0, \infty)$, we define the following energy:
\begin{align*}
  \mathbf{\Phi}^c: \mathcal{C}(M) &\to \mathbb{R},
  \\
  \Omega &\mapsto \mathbf{\Phi}(\Omega) - c\operatorname{Vol}(\Omega).
\end{align*}
and we will use the notation~\eqref{e:loc} for its localization to an open set \(U\). In case where $\partial\Omega$ is a smooth surface, the first variation of $\mathbf{\Phi}^c$ is given by
\begin{equation}\label{1stvFc}
  \delta_{\mathbf{\Phi}^c}\Omega(X)=\int_{\partial\Omega}(h_\Phi^{\partial \Omega}-c)\langle X, \nu \rangle \, d\mathcal H^n, \qquad \forall X\in \mathcal X(M),
\end{equation}
where $\nu$ and $h_\Phi^{\partial \Omega}$ denote, respectively, the outward unit normal on $\partial \Omega$ and the $\Phi$-anisotropic mean curvature of $\partial \Omega$ with respect to $\nu$. Note that $h_\Phi^{\partial \Omega}=-\langle H_\Phi^{|\partial \Omega|},\nu \rangle$, with the notation introduced in the previous section.

From~\eqref{1stvFc} we observe that any critical point $\Omega$ of $\mathbf{\Phi}^c$ satisfies  $h_\Phi^{\partial \Omega}\equiv c$. We can also compute the second variation $\delta^2_{\mathbf{\Phi}^c}\Omega(X,X)$ at a critical point $\Omega$, see~\cite[Appendix A.1]{FigalliMaggi},~\cite[Section 1.5, Page 295]{Allard} or~\cite[Lemma A.5]{GuidoFrancesco} for the explicit computation. As it is well known, the second variation along critical points only depends on the normal component \(\varphi=\langle X, \nu \rangle\) and we will abuse our notation by writing $\delta^2_{\mathbf{\Phi}^c}\Omega(\varphi,\varphi)$. The formula obtained for $\delta^2_{\mathbf{\Phi}^c}\Omega(\varphi,\varphi)$ at a critical point $\Omega$ does not depend on $c$, and it can also be applied to oriented hypersurfaces $\Sigma$ that are not necessarily closed. In such cases, we will use the notation $S^\Sigma_\Phi(\varphi, \varphi)$.

\begin{definition}
Let $\Sigma$ be an immersed, smooth, two-sided hypersurface  with unit normal vector $\nu$, and let $U\subset M$ be an open set. We will say that $\Sigma$ is a \emph{$c$-stable hypersurface} in $U$ if the anisotropic mean curvature with respect to $\nu$ satisfies $-\langle H_\Phi^{|\partial \Omega|}, \nu \rangle=h_\Phi^{\Sigma}\equiv c$ in $U$, and $S^\Sigma_\Phi(\varphi, \varphi)\geq 0$  for all  $\varphi\in C^\infty(\Sigma)$ with $\operatorname{spt} (\varphi)\subset \Sigma\cap U$. If $c=0$, we simply say that $\Sigma$ is a  stable hypersurface.
\end{definition}

Although the precise  formula of the second variation is important in~\cite{Allard}, on which the compactness Theorem~\ref{T:compact} is based, in the rest of the proof we never need its explicit formula, but only the following consequence, cf.~\cite[Lemma A.5]{GuidoFrancesco}:  if \(\Sigma\) is \(c\)-stable in \(U\), there exists a constant \(C=C(M,g,\lambda)\) such that
 \begin{equation}
     \label{eq:stability}
     \int_\Sigma \varphi^2|A_\Sigma|^2\le C\int_\Sigma |\nabla \varphi|^2+\varphi^2 \qquad \text{for all \(\varphi\in C_c^1(\Sigma\cap U)\)},
 \end{equation}
 where \(A_\Sigma\) denotes the second fundamental form of \(\Sigma\).

We will also need the notion of \emph{almost embedded minimal surface}, as introduced in~\cite{ZhouZhu2019}.
\begin{definition}
  Consider an open subset $U\subset M$, and a smooth \(n\) dimensional (possibly open) manifold  $\Sigma$. A smooth immersion $\psi: \Sigma\rightarrow U$ is an \emph{almost embedding} if, for every $p\in\psi(\Sigma)$ where $\psi$ fails to be an embedding, there exists a neighborhood $B\subset U$ of $p$, such that
  \begin{itemize}
  \item $\Sigma\cap \psi^{-1}(B)=\cup_{i=1}^k \Sigma_i$, where $\Sigma_i$ are disjoint connected components;
  \item $\psi(\Sigma_i)$ is an embedding for every $i=1,\dots,k$;
  \item for each $i$, every $\psi(\Sigma_j)$, with $j\neq i$, lies on one side of $\psi(\Sigma_i)$ in $B$.
  \end{itemize}
  In this case, we identify $\psi(\Sigma)$ with $\Sigma$ and $\psi(\Sigma_i)$ with $\Sigma_i$, and we say that $\Sigma$ is almost embedded.
\end{definition}

For a smooth almost embedded $c$-stable hypersurface $\Sigma$, we define the \textit{touching set} $\mathcal{S}(\Sigma)$ of $\Sigma$ as the set of points of $\Sigma$ where it fails to be embedded. By the maximum principle argument in~\cite[Lemma 2.7]{ZhouZhu2019}, \(\mathcal{S}(\Sigma)\) is locally the union  of two tangential smooth hypersurfaces with constant anisotropic mean curvature. We further denote with $\mathcal{R}(\Sigma)$ the set of points of $\Sigma$ where $\Sigma$ is locally smooth and embedded.

We define the regular set $\operatorname{Reg}(\Sigma):=\mathcal{S}(\Sigma)\cup \mathcal{R}(\Sigma)$. From the proof of~\cite[Lemma 5.1]{DePhilippis_DeRosa2024}, we deduce that $\mathcal{S}(\Sigma)$ can be covered with a finite number of balls, in which $\mathcal{S}(\Sigma)$ is an $(n-1)$-dimensional $C^1$  graph in $\Sigma$. We denote the singular set by $\operatorname{Sing}(\Sigma):=\overline \Sigma \setminus \operatorname{Reg}(\Sigma)$.

An important class of sets we will be dealing with in the sequel are those that satisfy the following property:
\begin{definition}[mass ratio upper bound]
  A Caccioppoli set $\Omega \in \mathcal{C}(M)$ is said to have a \emph{mass ratio upper bound} $C \in (0, \infty)$, if, for any $p \in M$, $r \in (0, \mathrm{inj}(M)/2)$, the following holds:
  \begin{equation*}
    \mathbf{\Phi}(\Omega;B_r(p)) \leq C r^n\,.
  \end{equation*}
\end{definition}
Note that, in view of~\eqref{limitato}, this is equivalent to \(\Per(\Omega;B_{r}(p))\lesssim r^{n}\).

A key role will be played by the following compactness theorem, which is an easy consequence of Theorem~\ref{compactness} below.

\begin{theorem}\label{T:compact}
  Given an open set \(U\) and a sequence of almost embedded, $c_k$-stable hypersurfaces $\Sigma_k\subset U$ having a mass ratio upper bound $C$, and such that
  \begin{itemize}
  \item[-] $\sup_{k} \mathcal{H}^n(\Sigma_k) < \infty$,
  \item[-] $\mathcal{H}^{n-2}(\operatorname{Sing}(\Sigma_k)) =0$,
  \item[-]  $\sup_k c_k <\infty$.
  \end{itemize}
  Then the following hold:
  \begin{itemize}
  \item[(i)] if $\inf c_k>0$, then $\{\Sigma_k\}$ converges locally smoothly to an almost embedded $c$-stable hypersurface $\Sigma$ in $U$ (for some $c>0$) with $\mathcal H^{n-2}(\Sing(\Sigma))=0$, after possibly passing to a subsequence; moreover, if $\{\Sigma_k\}$ are all boundaries, then the density of $\Sigma$ is $1$ on $\mathcal{R}(\Sigma)$ and $2$ on $\mathcal{S}(\Sigma)$, and $\Sigma$ is a boundary as well;
  \item[(ii)] if $c_k\to 0$, then $\{\Sigma_k\}$ converges locally smoothly with integer multiplicity to an embedded stable hypersurface $\Sigma$ in $U$  with $\mathcal H^{n-2}(\Sing(\Sigma))=0$, after possibly passing to a subsequence.
  \end{itemize}
\end{theorem}
\begin{proof}
  The proof is the same of~\cite[Theorem 2.11]{Zhou-zhu2020}, replacing the use of~\cite[Theorem 2.6]{Zhou-zhu2020} with Theorem~\ref{compactness}.
\end{proof}

We conclude by recalling the following:
\begin{definition}[local minimizing property]
  Given $c \in [0, \infty)$ and an open subset $U \subset M$, a Caccioppoli set $\Omega \in \mathcal{C}(M)$ is said to be locally $\mathbf{\Phi}^c$-minimizing in $U$ if, for any $p \in U$, there exists a geodesic ball $B_r(p) \subset U$ such that for any $\tilde \Omega  \in \mathcal{C}(M)$ with $\widetilde{\Omega} \Delta \Omega \subset B_r(p)$, it holds
  \[
    \mathbf{\Phi}^c(\Omega) \leq \mathbf{\Phi}^c(\widetilde{\Omega}).
  \]
\end{definition}

\section{Sweepouts}\label{sec:sweepouts}

In this section we introduced the key notion of sweepout. Throughout, we will assume that  $\Phi$ satisfies~\eqref{limitato}, and that $c$ is a constant in $[0, \infty)$.  We start by noticing that the reduced boundary map $\partial^*: \mathcal{C}(M) \to \mathcal{Z}^0_n(M; \mathbb{Z}_2)$ induces a double cover. We then give the following definition:

\begin{definition}[sweepout]
  A \emph{sweepout} on $M$ is a continuous map $\Gamma: [0, 1] \to \mathcal{Z}^0_n(M; \mathbb{Z}_2)$  satisfying  the following conditions:
  \begin{enumerate}[\normalfont(1)]
  \item There exists a continuous map $\Omega: [0,1] \to \mathcal{C}(M)$ such that $\Omega(0) = 0$, $\Omega(1) = M$, and $\Gamma(t) = \partial^* \Omega(t)$ for all $t \in [0,1]$.
  \item There is no concentration of mass,  meaning that
    \begin{equation}\label{eqn:no_mass_conc}
      \lim_{r \to 0} \sup\bigl\{\mathbf{\Phi}(\Gamma(t); B_r(p)) \mid t \in [0, 1], p \in M\bigr\} = 0\,.
    \end{equation}
  \end{enumerate}
  We will often use the shorthand notation $\{\Gamma_t = \partial^* \Omega_t\}_{t \in [0,1]}$, where $\Gamma_t = \Gamma(t)$ and
$\Omega_t = \Omega(t)$. The collection of all sweepouts of $M$ is denoted by $\Pi(M)$ or $\Pi$.
\end{definition}

We can now define the \(\Phi^{c}\)-width of \(M\).
\begin{definition}[width]
  For every $c \in [0, \infty)$ we define  $W^c_\Phi(M, g)$ as
  \[
    W^c_\Phi(M, g) \equiv \inf_{\Gamma \in \Pi(M)}\sup_{t \in [0,1]} \mathbf{\Phi}^c(\Omega(t))\,.
  \]
  We will often  write $W^c_\Phi$ without mentioning $(M, g)$.
\end{definition}

The following is an easy consequence of the isoperimetric inequality; see~\cite[Proposition~3.1-Remark 3.2]{DePhilippis_DeRosa2024} for a proof.
\begin{lemma}\label{lemma:utiledopo}
  For any $c \in [0, \infty)$, there exists a positive constant $C = C(M, \Phi, c)$, such that
   \[
    W^c_\Phi \geq C(M, \Phi, c) > 0\,.
  \]
  Moreover,
  \[
    W^c_\Phi(M, g)\leq \inf_{\Gamma  \in \Pi(M)}\sup_{t \in [0,1]} \mathbf{\Phi}(\Gamma(t))=W^0_\Phi(M, g)<\infty \qquad \forall c \geq 0\,.
  \]
\end{lemma}

\subsection{ONVP sweepouts \`a la Chodosh-Liokumovich-Spolaor}
 A key role will be played by the notion of \emph{Optimal Nested Volume-Parameterized} (\emph{ONVP}) sweepout, which was was first introduced in~\cite{Chodosh-Liokumovich-Spolaor2022} to study singular behaviors of $1$-width min-max minimal hypersurfaces in an $8$-dimensional closed Riemannian manifold. Here, we adapt this notion to our setting.
\begin{definition}[anisotropic ONVP sweepout]
  A sweepout $\{\Gamma_t = \partial^* \Omega_t\}_{t \in [0,1]}$ is an anisotropic \emph{optimal nested volume parameterized (ONVP) sweepout} for $(\Phi, c)$ if it satisfies the following conditions:
  \begin{description}
  \item [Optimal] $\sup_{t \in [0, 1]} \mathbf{\Phi}^c(\Omega_t) = W^c_\Phi(M, g)$;
  \item [Nested] $\Omega_{t_1} \subset \Omega_{t_2}$ for all $0 \leq t_1 \leq t_2 \leq 1$;
  \item [Volume-Parameterized] $\mathrm{Vol}(\Omega_t) = t\ \mathrm{Vol}(M, g)$ for all $t \in [0, 1]$.
  \end{description}
\end{definition}

The following theorem ensures the existence of ONVP sweepouts.
\begin{theorem}[existence of anisotropic ONVP sweepouts] \label{thm:exist_ONVP}
  Let $(M, g)$ be a closed Riemannian manifold, $c \ge 0$, and $\Phi$ be an elliptic integrand. Then there exists an anisotropic ONVP sweepout for $(\Phi, c)$.
\end{theorem}
\begin{proof}
  The proof is essentially the same as that of~\cite[Proposition~1.4]{Chodosh-Liokumovich-Spolaor2022}, with the only difference being that we replace~\cite[Theorem 1.4]{Chambers-Liokumovich2020} with the nested approximation Lemma~\ref{lem:nested_approx} below.

  Let $\{\Gamma^i \}^\infty_{i=1}$ be a sequence of sweepouts such that
  \[
    \lim_{i \to \infty} \sup_{t \in [0, 1]} \mathbf{\Phi}^c(\Omega^i_t) = W^c_\Phi(M, g)\,.
  \]
  For each $\Gamma^i$, we apply the nested approximation Lemma~\ref{lem:nested_approx} with $\varepsilon = \frac{1}{i}$ to obtain a nested continuous map $\tilde \Gamma_{t}^i = \partial \tilde \Omega_{t}^i$. Since $\tilde \Omega_0 \subset \Omega_0 = \emptyset$ and $M = \Omega_1 \subset \tilde \Omega_1$, we have $\tilde \Omega_0 = \emptyset$ and $\tilde \Omega_1 = M$ and thus, $\tilde \Gamma^i$ is also a sweepout. Let $\varphi^i(t) \coloneqq \mathrm{Vol}(\tilde \Omega_t)$, which is strictly increasing, so it has a continuous inverse $(\varphi^i)^{-1}$. We define a new sweepout $\hat \Gamma^i$ by
  \[
    \hat \Gamma^i(t) \coloneqq \tilde \Gamma^i \circ (\varphi^i)^{-1}(t\ \mathrm{Vol}(M, g))\,.
  \]

  Note that $\{\hat \Gamma^i\}^\infty_{i = 1}$ is a sequence of Nested Volume-Parameterized sweepouts with
  \[
    \lim_{i \to \infty} \sup_{t \in [0,1]}\mathbf{\Phi}^c({\hat \Omega}^i_t) = W^c_\Phi(M,g)\,.
  \]
  Moreover, since the sweepout is nested and volume parametrized, for \(s<t\),
  \[
  \Vol(\hat \Omega^i_t\Delta \hat \Omega^i_s)=  \Vol(\hat \Omega^i_t\setminus \hat \Omega^i_s)=(t-s) \Vol(M, g).
  \]
  By the Arzel\`a-Ascoli theorem, there exists a subsequence that converges to a sweepout $\Gamma'$, which is an anisotropic ONVP sweepout for $\Phi$.
\end{proof}

\begin{lemma}\label{lem:nested_approx}
  Suppose that $\Gamma = \partial^* \Omega$ is a continuous map on $[0, 1]$ in the sense that $\Gamma: [0, 1] \to \mathcal{Z}^0_n(M; \mathbb{Z}_2)$ and $\Omega: [0,1] \to \mathcal{C}(M)$ are continuous maps. Assume that $\sup_{t \in [0,1]} \mathbf{\Phi}^c(\Omega(t)) < \infty$. Then  for any $\varepsilon > 0$, there exists a continuous map $\tilde \Gamma = \partial^* \tilde \Omega$ defined on $[0, 1]$, such that
  \begin{enumerate}[\normalfont(1)]
  \item $\tilde \Omega(s) \subset \tilde \Omega(t)$ for all $0 \leq s \leq t \leq 1$ (i.e.\ $\tilde \Omega$ is nested);
  \item $\Vol(\Omega(t))$ is a strictly increasing function in $t$;
  \item $\tilde\Omega(0) \subset \Omega(0), \Omega(1) \subset \tilde \Omega(1)$;
  \item $\sup_{t \in [0, 1]}\mathbf{\Phi}^c(\tilde{\Omega}(t)) \leq \sup_{t \in [0, 1]} \mathbf{\Phi}^c(\Omega(t)) + \varepsilon$.
  \end{enumerate}
\end{lemma}
\begin{proof}[Sketch of Proof]
  The proof is essentially the same as that of~\cite[Proposition 6.1]{Chambers-Liokumovich2020}. Since $\mathbf{\Phi}$ and the area functional are absolutely continuous with respect to each other, we only need to replace $\mathcal{H}^n$ by $\mathbf{\Phi}$ loc.sit., and all the arguments follow.
\end{proof}

\begin{definition}[min-max sequence]
  Given an optimal  sweepout $\Gamma = \partial^* \Omega$ for $(\Phi, c)$, a \emph{min-max sequence} is a converging sequence $t_i \to t$ such that the following limit exists
  \[
    V = \lim_{i \to \infty} |\Gamma_{t_i}|\,,
  \]
  and
  \[
    W^c_\Phi(M, g) = \lim_{i \to \infty} \mathbf{\Phi}^c(\Omega(t_i))\,.
  \]
\end{definition}

\begin{definition}[critical domain and critical set]
  The \emph{critical domain} of an optimal  sweepout $\Gamma$ for $(\Phi, c)$ is defined as
  \[
    \mathbf{m}(\Gamma) \coloneqq \{t\in [0, 1]: \exists \text{ a min-max sequence } t_i \to t\}\,.
  \]
  The \emph{critical set} of an optimal  sweepout $\Gamma$ for $(\Phi, c)$ is defined as
  \[
    \mathbf{C}(\Gamma) \coloneqq \{V = \lim_{i \to \infty} |\Gamma_{t_i}|\in \mathcal{V}(M): \text{for some min-max sequence } t_i \to t\in [0,1]\}\,.
  \]
\end{definition}

\subsection{ONVP sweepouts with uniform mass ratio upper bounds}
In this subsection, our goal is to construct an ONVP sweepout that has a uniform mass ratio upper bound. This condition is crucial for our subsequent blowup analysis and curvature estimates.

\begin{theorem}[existence of anisotropic ONVP sweepouts with uniform mass ratio upper bound]\label{thm:exist_ONVP_bound}
  Fix $\Lambda > 0$. Let $(M^{n+1}, g)$ be a closed Riemannian manifold, let $c \in [0, \Lambda]$, and let $\Phi$ be an elliptic integrand satisfying~\eqref{limitato}. There exists an anisotropic ONVP sweepout $\Gamma = \partial^* \Omega$ for $(\Phi, c)$ and a constant $C = C(M, g, \lambda, \Lambda) > 0$, such that for any $t \in [0, 1]$, $\Omega(t)$ has a uniform mass ratio upper bound $C$:
  \begin{equation}\label{eqn:mass_ratio_bound}
    \mathbf{\Phi}(\Gamma(t); B_r(p)) \leq C r^n\,,
  \end{equation}
  for any $p \in M$, $r \in (0, \mathrm{inj}(M) / 2)$.  In the following, we will refer to $\Gamma$ as an anisotropic ONVP sweepout \emph{with uniform mass ratio upper bound}.
\end{theorem}

To prove the theorem, we will inductively modify the sweepout to ensure that the required uniform mass ratio upper bound holds for progressively smaller radii $r$, while preserving the ONVP property. By taking a subsequential limit of these modified sweepouts, we obtain an ONVP sweepout with uniform mass ratio upper bound at all scales.

\subsubsection{Triangulation}
It is a well-known fact that for any $\varepsilon > 0$, there exists \(\mu>0\) such that the closed Riemannian manifold $M^{n+1}$ can be triangulated in such a way that, if each simplex is equipped with the standard Euclidean metric with the same edge-length $\mu > 0$, then the resulting metric on $M$ is $(1 + \varepsilon)$-bilipschitz to $g$. Let us fix an $\varepsilon \in (0, 1)$ with $(1 + \varepsilon)^{n + 1} < 2$, and choose such a triangulation on $M$.

We begin the inductive construction of $T_k$ with $T_0$, which is defined as the collection of the (closed) underlying sets of all simplexes of dimension $(n+1)$ in the triangulation constructed above. To define $T_{k+1}$ for $k \in \mathbb{N}$, we subdivide each simplex in $T_k$ edgewise, as in~\cite{Edelsbrunner-Grayson2000}, and denote the set of the (closed) underlying sets of all new simplexes of dimension $(n+1)$ by $T_{k+1}$.

We start with the following geometric lemma whose proof is postponed to Appendix \ref{sec:geometric}.
\begin{lemma}\label{lem:trig_number}
  Let $(M^{n+1}, g)$ be a closed Riemannian manifold and $\{T_k\}^\infty_{k = 0}$ be a sequence of triangulations as described above with constants $\varepsilon \in (0, 1)$ and $\mu > 0$. Then there exists a positive integer $C = C(M, g)$, such that for any $r\in (0,\mu]$ and any $p \in M$, $B_r(p)$ intersects with at most $C$ simplices in $T_k$ where $k$ satisfies $2^{-k} \mu \leq r < 2^{-k+1} \mu$.
\end{lemma}

Let $\sigma$ be an \emph{open} regular simplex in $\mathbb{R}^{n+1}$ centered at $0$ with edge-length $\mu > 0$. Without loss of generality, we can assume that one of its faces is parallel to the hyperplane
\[
\{x_{n + 1} = 0\} = \{(x_1, x_2, \ldots, x_{n+1}) \in \mathbb{R}^{n+1} \mid x_{n+1} = 0\}
\]
and that for some $a < 0 < b$, this face is given by $\overline{\sigma} \cap \{x_{n + 1} = a\}$ and the vertex opposite to this face is given by $\bar \sigma \cap \{x_{n+1} = b\}$. We define a continuous map
\begin{equation}\label{eqn:tri_cts_map}
  \tilde \sigma: [0, 1] \to \mathcal{C}(\bar \sigma), \quad \tilde \sigma(t) =\bar \sigma \cap \{x_{n+1} \leq a + (b - a)t\}\,.
\end{equation}
Note that the intersection has a bound on its Euclidean volume
\begin{equation}\label{eqn:tri_vol}
  \mathcal{H}^n (\partial \tilde \sigma(t) \cap \bar \sigma) \leq c_n \mu^n\quad \forall t \in [0, 1]\,,
\end{equation}
where $c_n$ is the Euclidean volume of a regular $n$-simplex with edge length $1$.

\subsubsection{Proof of Theorem~\ref{thm:exist_ONVP_bound}}

Let $\Gamma = \partial^* \Omega$ be an anisotropic ONVP sweepout as in Theorem~\ref{thm:exist_ONVP}, and let $\{T_k\}^\infty_{k = 0}$ be a sequence of triangulations with $\varepsilon$ and $\mu$ as in Lemma~\ref{lem:trig_number}. Furthermore, we can choose $\mu$ small enough such that
\begin{equation}\label{eqn:mu_bound}
  c_n\mu^n > 10\Lambda c_{n+1}\mu^{n+1}.
\end{equation}
where $c_n$ is the Euclidean volume of a regular $n$-simplex with edge length $1$.

We write $\Gamma_{-1} = \Gamma$. We shall inductively construct a sequence of anisotropic ONVP sweepouts $\{\Gamma_k = \partial^* \Omega_k\}^\infty_{k = 0}$ such that for each $\Gamma_k$, the mass ratio upper bound
\begin{equation}
  \mathbf{\Phi}(\Gamma_k(t); B_r(p)) \leq C_1 r^n
\end{equation} holds for $r \in [2^{-k}\mu, \mu]$. Here, the constant $C_1 = C \cdot (20 (n+2)c_n \lambda)$, where $C$ is the constant in Lemma~\ref{lem:trig_number}.

For each $k \in \mathbb{N}$, suppose that $\Gamma_{k-1}$ has been constructed, such that for all $k' \in \{0, 1, 2, \cdots, k - 1\}$, $\sigma' \in T_{k'}$, and $t \in [0, 1]$,
\begin{equation}\label{eqn:mass_ratio_bound_2}
  \mathbf{\Phi}(\Omega_{k-1}(t); \sigma') - c\Vol(\Omega_{k-1}(t) \cap \sigma') \leq C_2 \left(2^{-k'}\mu\right)^n\,,
\end{equation}
where
\begin{equation}\label{e:C2}
C_2 := 12(n+2) c_n \lambda.
\end{equation}

We fix a $\sigma \in T_k$, and it follows from the lower semi-continuity of $\mathbf{\Phi}$ that the set of ``bad'' slices
\begin{equation}\label{e:bsigma}
  \mathcal{B}_\sigma \coloneqq \{t \in [0,1] \mid  \mathbf{\Phi}(\Gamma(t); \sigma) > C_2 \left(2^{-k}\mu\right)^n\}\,,
\end{equation}
is an open subset of $[0, 1]$.
Hence, $\mathcal{B}_\sigma$ is an (at most) countable union of disjoint open intervals
\[
  \mathcal{B}_\sigma = \coprod^\infty_{i = 1} (a_i, b_i)\,.
\]

If $\mathcal{B}_\sigma = \emptyset$, then we are done and set $\Gamma_k = \Gamma_{k - 1}$. Otherwise, we modify $\Gamma_{k - 1}$ in $\sigma$ and inductively construct a sequence $\{\Gamma^{j}_{k - 1}\}^\infty_{j=1}$ with
\[
    \Gamma^j_{k - 1} \cap (M \setminus \bar\sigma) = \Gamma_{k - 1} \cap (M \setminus \bar\sigma)\,.
\]
We let $\tilde \Gamma^j_{k - 1}$ be volume-parametrized $\Gamma^j_{k - 1}$. In the following, for $j \in \mathbb{N}^+$, we use the notations:
\[
  \mathcal{B}^j_\sigma \coloneqq \{t \in [0,1] \mid \mathbf{\Phi}(\Gamma_{k-1}^{j}(t); \sigma)  > C_2 \left(2^{-k}\mu\right)^n\}\,,
\]
and
\[
  \tilde{\mathcal{B}}^j_\sigma \coloneqq \{t \in [0,1] \mid \mathbf{\Phi}(\tilde\Gamma_{k-1}^{j}(t); \sigma)  > C_2 \left(2^{-k}\mu\right)^n\}\,.
\]

\paragraph*{\textbf{Step} $1$}
We shall construct $\Gamma^1_{k-1}$ from $\Gamma_{k-1}$. For technical reasons, we need to consider an exhaustion \(\{\sigma_{m}\}_{m \in \mathbb{N}^+}\) of \(\sigma\) such that
\[
\overline \sigma_{m}\subset \sigma\qquad  \text{and}\qquad  \sigma=\bigcup_{m}\sigma_{m}.
\]
Note that these can be constructed so that
\begin{equation}\label{sigma}
\Vol(\sigma_{m})\le \Vol(\sigma) \qquad \ \mathcal H^{n}(\partial \sigma_{m}) \le 1.01 \cdot \mathcal H^{n}(\partial \sigma).
\end{equation}

Fix  \(m\in \mathbb N^+\) and let  $E^{m}_{a_1}$ be a Caccioppoli set such that
\begin{enumerate}
\item $E^{m}_{a_1} \Delta \Omega_{k - 1}(a_1) \subset \sigma_{m} \cap \Omega_{k - 1}(a_1)$;
\item For all $E'$ with $E' \Delta \Omega_{k - 1}(a_1) \subset \sigma_{m} \cap \Omega_{k - 1}(a_1)$
\[
\mathbf{\Phi}^c(E^{m}_{a_1}) \leq \mathbf{\Phi}^c(E').
\]
\end{enumerate}
The existence of $E^{m}_{a_1}$ immediately follows from a straightforward applications of the direct method in the calculus of variations. Indeed, one can take \(E^{m}_{a_{1}}\) as the largest (with respect to inclusion) of the solutions to the  minimization problem
\[
\min \Bigl\{ \mathbf{\Phi}^c(E)\mid   E\subset \Omega_{k-1}(a_1),\quad E\setminus \sigma_{m} =  \Omega_{k - 1}(a_1)\setminus \sigma_{m}\Bigr\}.
\]
By taking $E' = \Omega_{k-1}(a_1) \setminus \sigma_{m}$ in condition (2),  and using~\eqref{sigma} along with the fact that $\mathcal{H}^n(\partial \sigma) \leq 2(n+2)c_n (2^{-k} \mu)^n$, we obtain
\begin{equation}\label{eqn:E_a1_vol}
\begin{split}
  \mathbf{\Phi}(E^{m}_{a_1};\overline{\sigma_{m}}) & \leq 1.01 \cdot 2(n+2)c_n \lambda (2^{-k}\mu)^n + \Lambda c_{n+1}(2^{-k}\mu)^{n+1}
  \\
  &\leq 3(n+2)c_n \lambda (2^{-k}\mu)^n< \mathbf{\Phi}(\Omega_{k-1}^{t};\sigma)\,.
\end{split}
\end{equation}
where we have also  used~\eqref{eqn:mu_bound}.

Similarly, we let $E^{m}_{b_1}$ be the smallest Caccioppoli set (with respect to inclusion) such that
\begin{enumerate}
\item $E^{m}_{b_1} \Delta \Omega_{k - 1}(b_1) \subset \sigma_{m} \setminus \Omega_{k - 1}(b_1)$;
\item For all $E'$ with $E' \Delta \Omega_{k - 1}(b_1) \subset \sigma_{m} \setminus \Omega_{k - 1}(b_1)$
\[
\mathbf{\Phi}^c(E^{m}_{b_1}) \leq \mathbf{\Phi}^c(E').
\]
\end{enumerate}
By taking $E' = \Omega_{k-1}(b_1) \cup \sigma_{m}$ in condition (2) and arguing as in~\eqref{eqn:E_a1_vol}, we obtain that
\begin{equation}\label{eqn:E_b1_vol}
  \mathbf{\Phi}(E_{b_1}; \overline{\sigma_{m}}) \leq 3(n+2)c_n \lambda(2^{-k}\mu)^n\,.
\end{equation}

We now  define a new sweepout $\Gamma^{1,m}_{k - 1} (t)= \partial^* \Omega^1_{k - 1}(t)$ by
\[
  \Omega^{1,m}_{k - 1}(t) = \begin{cases}
    \Omega_{k-1}(t) \cap E^{m}_{a_1}, & t \leq a_1\\
    (\Omega_{k-1}(t) \setminus \sigma) \cup E^{m}_{a_1} \cup (E^{m}_{b_1} \cap \tilde \sigma(\frac{t - a_1}{b_1 - a_1})), & a_1 < t < b_1\\
    \Omega_{k-1}(t) \cup E^{m}_{b_1}, & t \geq b_1
  \end{cases}
\]
where $\tilde \sigma(t)$ is defined in~\eqref{eqn:tri_cts_map}.

\begin{figure}[!ht]
    \centering
    \includegraphics[width=\linewidth]{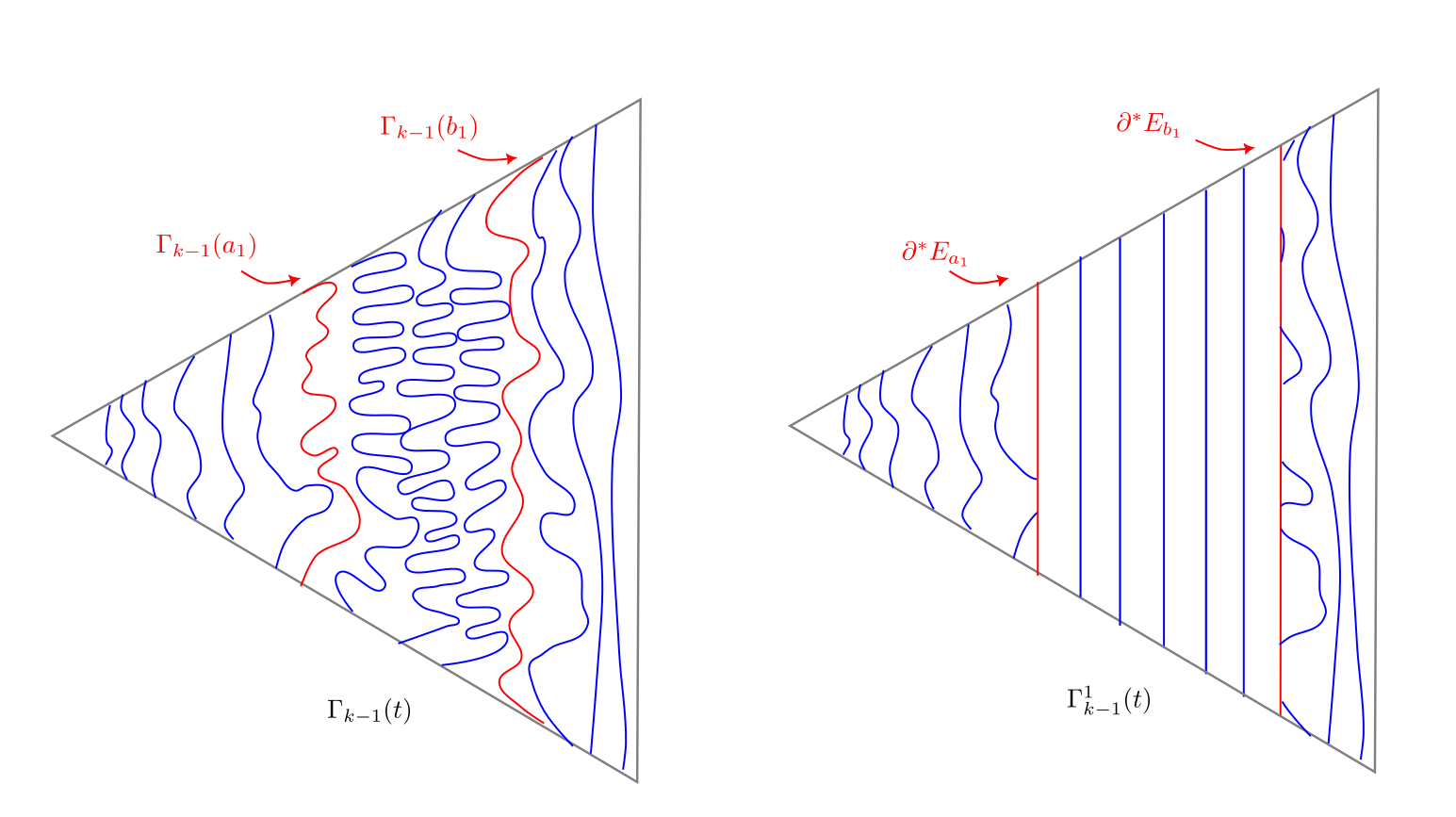}
     \caption{Construction of $\Gamma^1_{k - 1}(t)$}
   \label{fig:enter-label}
\end{figure}

From~\eqref{eqn:tri_vol},~\eqref{eqn:mu_bound},~\eqref{eqn:E_a1_vol} and~\eqref{eqn:E_b1_vol}, we obtain that for all $t \in (a_1, b_1)$,
\begin{equation}\label{eqn:est_good_slices_1_al}
  \mathbf{\Phi}(\Omega^{1,m}_{k-1}(t);  \overline{\sigma_{m}}) \leq 10(n+2)c_n \lambda (2^{-k}\mu)^n\,,
\end{equation}
since
\[
\Gamma^{1,m}_{k-1}(t) \cap \overline{\sigma_{m}} \subset \partial \sigma_{m} \cup (\partial^* E_{a_1}\cap \sigma_{m}) \cup (\partial^* E_{b_1} \cap \sigma_{m}) \cup \Bigl(\partial \tilde \sigma\Bigl(\frac{t-a_1}{b_1 - a_1}\Bigr) \cap \sigma\Bigr).
\]
In particular, by the definition of \(\mathcal B_{\sigma}^{1}\) and recalling~\eqref{e:C2}, we find that for  \(t\in (a_{1},b_{1})\)
\begin{equation}\label{eqn:est_slices_al}
\begin{aligned}
  &\mathbf{\Phi}^c(\Omega^{1,m}_{k-1}(t);\overline {\sigma_{m}})
  \\
  \le& \mathbf{\Phi}(\Omega^{1,m}_{k-1}(t);\overline {\sigma_{m}})
   \le  10(n + 2)c_{n}\lambda(2^{-k}\mu)^{n} &\text{by~\eqref{eqn:est_good_slices_1_al}}&
  \\
  \le& 10(n + 2)c_{n}\lambda(2^{-k}\mu)^{n}+\Lambda c_{n+1}(2^{-k}\mu)^{n+1}-c \Vol(\Omega_{k-1}(t)\cap \sigma)
  \\
  \le& 11(n + 2)c_{n}\lambda(2^{-k}\mu)^{n}-c \Vol(\Omega_{k-1}(t)\cap \sigma) &\text{by~\eqref{eqn:mu_bound}}&
  \\
  \le&  \mathbf{\Phi}(\Omega_{k-1}(t);\sigma)-c \Vol(\Omega_{k-1}(t)\cap \sigma)= \mathbf{\Phi}^c(\Omega_{k-1}(t); \sigma) &\text{by~\eqref{e:bsigma} and~\eqref{e:C2}}&.
\end{aligned}
\end{equation}
By (2)  in the construction  of $E^{m}_{a_1}$ and the inclusion  \(\Omega_{k-1}(t)\subset \Omega_{k-1}(a_{1})\) for $t \leq a_1$, we also obtain
\begin{equation}\label{eqn:est_slices_al<_a_1}
\begin{split}
  \mathbf{\Phi}(\Omega^{1,m}_{k-1}(t); \overline{\sigma_{m}}) - c \Vol(\Omega^1_{k-1}(t) \cap \sigma_{m})& \leq \mathbf{\Phi}(\Omega_{k-1}(t); \overline {\sigma_{m}}) - c \Vol(\Omega_{k-1}(t) \cap \sigma_{m})
  \\
  &\le \mathbf{\Phi}^{c}(\Omega_{k-1}(t); \sigma)+c\Vol(\sigma\setminus \sigma_{m}).
\end{split}
\end{equation}
Similarly,  for $t \geq b_1$,
\begin{equation}\label{eqn:est_slices_al>_b_1}
  \mathbf{\Phi}(\Omega^{1,m}_{k-1}(t); \overline {\sigma_m}) - c \Vol(\Omega^1_{k-1}(t) \cap \sigma_{m}) \leq \mathbf{\Phi}^{c}(\Omega_{k-1}(t); \sigma)+c\Vol(\sigma\setminus \sigma_{m}).
\end{equation}
We now note that  for \(\sigma' \in T_{k'}\), \(k'\in\{0,\dots,k-1\}\), and \(t\in [0,1]\),
\begin{equation}
\label{e:quasi}
\begin{split}
  \mathbf{\Phi}^c(\Omega^{1,m}_{k-1}(t); \sigma')& \le   \mathbf{\Phi}^c(\Omega_{k-1}(t);  \sigma')+c\Vol(\sigma\setminus \sigma_m)+\ \mathbf{\Phi}(\Omega^{1,m}_{k-1}(t); \sigma\setminus \overline{\sigma_m})
  \\
  &=  \mathbf{\Phi}^c(\Omega_{k-1}(t);  \sigma')+c\Vol(\sigma\setminus \sigma_m)+\ \mathbf{\Phi}(\Omega_{k-1}(t); \sigma\setminus \overline{\sigma_m}).
  \end{split}
\end{equation}
Indeed, since \(\Omega^{1,m}_k(t)\) coincides with \(\Omega_k(t)\) outside \(\sigma_m\Subset \sigma'\), the inequality above is trivial if  \(\sigma'\cap \sigma_m=\emptyset\). If  \(\sigma_m\Subset \sigma'\), the inequality follows from~\eqref{eqn:est_slices_al},~\eqref{eqn:est_slices_al<_a_1}, and~\eqref{eqn:est_slices_al>_b_1} by  decomposing \(\sigma'\) into \(\tilde \sigma \in T_k\) and exploiting the additivity  of \(\mathbf \Phi^c \) (note that one of the simplexes in the decomposition must coincide  with \(\sigma\)). The equality follows from the fact that \(\Omega_{k-1}^{1,m}\) equals \(\Omega_{k-1}\) one the open set \(\sigma\setminus \overline{\sigma_m}\). We remark  that, since \(\sigma_m \Subset \sigma\) for \(\sigma \in T_k\), these are the only two possibilities (recall that we are dealing with open simplexes). This is  the reason why we have introduced the addition parameter \(m\).

We now aim to pass to the limit as \(m\to \infty\). To this end,  note that,  by construction:
\[
\Vol(\Omega^{1,m}_{k-1}(t)\Delta \Omega^{1,m}_{k-1}(s))\le  \Vol(\Omega_{k-1}(t)\Delta \Omega_{k-1}(s))\le \Vol(M)|t-s|
\]
for \(t,s\) either in \([0,a_1]\) or in \([b_1,1]\), and there exists \(C\), independent of \(m\), such that
\[
\Vol(\Omega^{1,m}_{k-1}(t)\Delta \Omega^{1,m}_{k-1}(s))\le C|t-s|
\]
for \(t,s\) in \([a_1,b_1]\). Whence,  by the Arzel\`a-Ascoli theorem,  \(\Omega^{1,m}_{k-1}\) converges (up to a subsequence) to a sweepout \(\Omega^{1}_{k-1}\). By passing to the limit in~\eqref{e:quasi} and noticing that
\[
\Vol(\sigma\setminus \sigma_m)+\ \mathbf{\Phi}(\Omega_{k-1}(t); \sigma\setminus \overline{\sigma_m}) \to 0
\]
since \( \sigma\setminus \overline{\sigma_m}\downarrow \emptyset\), we  get that
\begin{equation}\label{e:ineq}
  \mathbf{\Phi}^c(\Omega^1_{k-1}(t); \sigma') \le   \mathbf{\Phi}^c(\Omega_{k-1}(t);  \sigma')
\end{equation}
for all \(\sigma'\in T_{k'}\), \(k'\in \{0,\dots, k-1\}\). In particular \(\Omega^{1}_{k-1}\) satisfies~\eqref{eqn:mass_ratio_bound_2} as well. Furthermore, since \(\sigma_m\uparrow\sigma\),
\[
\begin{split}
\mathbf\Phi^{c}(\Omega_{k - 1}^{1}(t);\sigma)&=\limsup_{m \to \infty} \mathbf\Phi^c(\Omega^{1, m}_{k - 1}(t); \sigma)\\
&\le \limsup_{m \to \infty} \mathbf\Phi^c(\Omega^{1, m}_{k - 1}(t); \overline{\sigma_m}) + \limsup_{m \to \infty} \mathbf\Phi(\Omega^{1,m}_{k - 1}(t); \sigma \setminus \overline{\sigma_m})\\
&= \limsup_{m \to \infty} \mathbf\Phi^c(\Omega^{1, m}_{k - 1}(t); \overline{\sigma_m}) + \limsup_{m \to \infty} \mathbf\Phi(\Omega_{k - 1}(t); \sigma \setminus \overline{\sigma_m})\\
&= \limsup_{m \to \infty} \mathbf\Phi^c(\Omega^{1, m}_{k - 1}(t); \overline{\sigma_m})\,.
\end{split}
\]
Hence,~\eqref{eqn:est_good_slices_1_al} implies that for $t \in (a_1, b_1)$,
\begin{equation}\label{eqn:est_good_slices_1}
  \mathbf{\Phi}(\Omega^{1}_{k-1}(t);  \sigma) \leq 10(n+2)c_n \lambda (2^{-k}\mu)^n\,,
\end{equation}
while the same argument that proved~\eqref{e:ineq} also shows that for all \(t\in [0,1]\),
\begin{equation}\label{e:ineq2}
\mathbf{\Phi}^c(\Omega^1_{k-1}(t))\le \mathbf{\Phi}^c(\Omega_{k-1}(t)).
\end{equation}
Inequality~\eqref{eqn:est_good_slices_1} implies  $\mathcal{B}^1_\sigma \subset
\mathcal{B}_\sigma \setminus (b_1, a_1)$, and therefore,
\[
  \mathcal{L}^1(\mathcal{B}^1_\sigma) \leq \mathcal{L}^1(\mathcal{B}_\sigma) - |b_1 - a_1|\,.
\]
Moreover, for $t_1, t_2 \leq a_1$,
\[
  \Vol(\Omega^1_{k - 1}(t_1) \Delta \Omega^1_{k - 1}(t_2)) \leq \Vol(\Omega_{k - 1}(t_1) \Delta \Omega_{k - 1}(t_2))\,,
\]
and for $t_1, t_2 \geq b_1$,
\[
  \Vol(\Omega^1_{k - 1}(t_1) \Delta \Omega^1_{k - 1}(t_2)) \leq \Vol(\Omega_{k - 1}(t_1) \Delta \Omega_{k - 1}(t_2))\,.
\]
Hence, the volume-parametrized sweepout $\tilde \Gamma^1_{k - 1}$ satisfies
\[
  \mathcal{L}^1(\tilde{\mathcal{B}}^1_\sigma) \leq \mathcal{L}^1(\mathcal{B}^1_\sigma)\,.
\]
Finally,~\eqref{e:ineq2} implies that $\tilde \Gamma^1_{k - 1}$ is an anisotropic ONVP sweepout.

\paragraph*{\textbf{Step} $j + 1$} Assuming that we have constructed $\Gamma^j_{k - 1} = \partial^* \Omega^j_{k - 1}$  satisfying~\eqref{eqn:mass_ratio_bound_2} and such that $\mathcal{B}^j_\sigma \subset \mathcal{B}_\sigma \setminus \bigcup^j_{i = 1}(b_i, a_i)$, we can repeat the previous construction to get a new sweepout $\Gamma^{j+1}_{k-1}$ still satisfying~\eqref{eqn:mass_ratio_bound_2}  and such that  $\mathcal{B}^{j + 1}_\sigma \subset \mathcal{B}_\sigma \setminus \bigcup^{j + 1}_{i = 1}(b_i, a_i)$ and
\[
  \mathcal{L}^1(\tilde{\mathcal{B}}^{j + 1}_\sigma) \leq \mathcal{L}^1(\mathcal{B}^{j + 1}_\sigma) \leq \mathcal{L}^1(\mathcal{B}_\sigma) - \sum^{j + 1}_{i=1} |b_i - a_i|\,.
\]
By the Arzel\`a-Ascoli theorem, we can pass to a subsequential limit as \(j \to \infty\), yileding ONVP sweepouts \(\tilde \Gamma_{k-1}^{j}\) that converge
 to an ONVP sweepout $\Gamma^\infty_{k - 1}$. By the lower semi-continuity of $\mathbf{\Phi}$, the ``bad'' slice set
\[
  \mathcal{B}^\infty_\sigma \coloneqq \{t \in [0,1] \mid \mathbf{\Phi}(\Gamma^\infty_{k - 1}(t); \sigma) > C_2 \left(2^{-k}\mu\right)^n\}\,
\]
satisfies
\[
  \mathcal{L}^1(\mathcal{B}^\infty_\sigma) \leq \limsup_{j \to \infty} \mathcal{L}^1(\tilde{\mathcal{B}}^j_\sigma) \leq \limsup_{j \to \infty} \mathcal{L}^1(\mathcal{B}^j_\sigma) = 0\,.
\]
Consequently, $\mathcal{B}^\infty_\sigma = \emptyset$. Hence, for all $t \in [0, 1]$,
\[
  \mathbf{\Phi}(\Gamma^\infty_{k - 1}(t);\sigma) \leq C_2 \left(2^{-k}\mu\right)^n\,,
\]
and $\Gamma^\infty_{k - 1} \setminus \bar \sigma$ coincides with $\Gamma_{k - 1} \setminus \bar \sigma$, up to reparameterization. By~\eqref{e:ineq} and the lower semi-continuity of $\mathbf{\Phi}$ again,  $\Gamma^\infty_{k - 1}$ also satisfies~\eqref{eqn:mass_ratio_bound_2}.

Since there are only finitely many $\sigma \in T_k$, we can inductively perform the replacement construction described above for each $\sigma \in T_k$ to finally obtain an anisotropic ONVP sweepout, denoted by $\Gamma_k = \partial^* \Omega_k$. This sweepout satisfies, for all $\sigma \in T_k$,
\[
  \mathbf{\Phi}(\Gamma_{k}(t);\sigma) \leq C_2 \left(2^{-k}\mu\right)^n\,,
\]
Together with the estimates~\eqref{eqn:mass_ratio_bound_2}, we conclude that for all $k' \in \{0, 1, \cdots, k\}$, $\sigma' \in T_{k'}$, and $t \in [0, 1]$,
\[
  \mathbf{\Phi}^c(\Gamma_{k}(t); \sigma') \leq C_2 \left(2^{-k'}\mu\right)^n\,,
\]
and thus, since $\mathcal{H}^n(\partial \sigma') \leq 2(n + 2)(2^{-k'}\mu)^n$, we have
\begin{equation}\label{eqn:mass_ratio_bound_3}
  \mathbf{\Phi}^c(\Gamma_{k}(t);\overline {\sigma'}) \leq (C_2 + 2(n+2)c_n\lambda) \left(2^{-k'}\mu\right)^n \leq (15(n+2)c_n\lambda) \left(2^{-k'}\mu\right)^n\,.
\end{equation}

By Lemma~\ref{lem:trig_number}, for any $p \in M$ and any $r \in [2^{-k}\mu, \mu]$, let $k' \in \{0, 1 \ldots, k\}$ such that $2^{-k'}\mu \leq r < 2^{-k' + 1}\mu$, there are at most $C$ many $\sigma \in T_{k'}$ that intersect with $B_r(p)$, and consequently, for any $t \in [0, 1]$,
\[
\begin{aligned}
  & \mathbf{\Phi}(\Gamma_k(t); B_r(p)) 
  \\
  \le& \sum_{\sigma \in T_{k'}, \sigma \cap B_r(p) \neq \emptyset} \mathbf{\Phi}(\Gamma_k(t);\overline \sigma)
  \\
  \le& \sum_{\sigma \in T_{k'}, \sigma \cap B_r(p) \neq \emptyset} \mathbf{\Phi}^c(\Gamma_k(t);\overline \sigma) + c\sum_{\sigma \in T_{k'}, \sigma \cap B_r(p) \neq \emptyset} \operatorname{Vol}(\sigma)
  \\
  \le& C \cdot (15(n+2)c_n\lambda) \left(2^{-k'}\mu\right)^n + 2\Lambda \cdot C \cdot c_{n + 1}\left(2^{-k'}\mu\right)^{n + 1} &\text{by~\eqref{eqn:mass_ratio_bound_3}}&
  \\
  \le& C_1 r^n &\text{by~\eqref{eqn:mu_bound}}&\,.
\end{aligned}
\]

Finally, letting \(k\to \infty\), the sequence $\{\Gamma_k \}^\infty_{k = 0}$ converges subsequenially to an ONVP sweepout, denoted by $\Gamma'$. It follows from the lower semi-continuity of $\mathbf{\Phi}$ again that for any $p \in M$, $r \in (0, \mu]$, and $t \in [0, 1]$,
\[
  \mathbf{\Phi}(\Gamma'(t);B_r(p)) \leq C_1 r^n
\]
Since the above estimate is clearly satisfied when \(r\ge \mu\) with a possibly larger constant, the proof is concluded.

\section{Almost minimizing varifolds  with  uniform mass ratio upper bound}
\label{sec: Alomost minimizing varifolds and unifrom mass ratio}
In this section, we fix a closed Riemannian manifold $(M^{n+1}, g)$ and an elliptic integrand $\Phi$ satisfying~\eqref{limitato}. Throughout the following, we will use the notation $\an(p,r,s) := B_s(p) \setminus \overline{B_r(p)}$.

We start with the following two (by now standard) definitions.

\begin{definition}[$c$-almost minimizing varifolds]\label{def:am_varifolds}
  For any given $\varepsilon, \delta > 0$ and any open subset $U \subset M$, we define $\mathfrak{a}^{\Phi,c}(U; \varepsilon, \delta)$ to be the set of all $\Omega \in \mathcal{C}(M)$ such that if $\{\Omega_i\}^m_{i = 0} \subset \mathcal{C}(M)$ is a sequence such that
  \begin{enumerate}[\normalfont(1)]
  \item $\Omega_0 = \Omega$;
  \item $\Omega_i \Delta \Omega \subset U$;
  \item \(\Vol(\Omega_i \Delta \Omega_{i - 1}) \leq \delta\);
  \item $\mathbf{\Phi}^c(\Omega_i) \leq \mathbf{\Phi}^c(\Omega) + \delta$, for all $i = 1, 2, \ldots, m$;
  \end{enumerate}
  then $\mathbf{\Phi}^c(\Omega_m) \geq \mathbf{\Phi}^c(\Omega) - \varepsilon$.

  We say that a varifold $V \in \mathcal{V}_n(M)$ is \emph{$c$-almost minimizing for $\Phi$ in $U$} if there exist sequences $\varepsilon_i \to 0$, $\delta_i \to 0$, and $\Omega_i \in \mathfrak{a}^{\Phi, c}(U;\varepsilon_i, \delta_i)$, such that $\mathbf{F}(|\partial^* \Omega_i|, V) \leq \varepsilon_i$ where  \(\mathbf{F}\) is the
 the canonical distance on varifolds.  When we need to  specify the sequence $\Omega_i$, we will say that $V$ is $c$-almost minimizing for $\Phi$ in $U$ by means of $\Omega_i$. If we further want to specify also the sequences $\varepsilon_i, \delta_i$, we will say that $V$ is $c$-almost minimizing for $\Phi$ in $U$ by means of $\Omega_i$, $\varepsilon_i, \delta_i$.
\end{definition}

The following concepts were first introduced in \cite{Li2023} to prove a compactness result of Almgren-Pitts width realizations. Here, for convenience, we adapt the notations from \cite{WZ2023}.

\begin{definition}\label{admissible_annuli}
  Given an $L \in \mathbb{N}^+$ and $p \in M$, a collection of annuli centered at $p$
  \[
    \mathscr{C} := \{\an(p, s_1, r_1), \an(p, s_2, r_2), \cdots, \an(p, s_L, r_L)\}
  \]
  is called \emph{$L$-admissible} if $2r_{j + 1} < s_j$ for all $j = 1, 2, \cdots, L - 1$.
\end{definition}

\begin{definition}\label{def:am_admissible_annuli}
  Given an $L$-admissible collection of annuli $\mathscr{C}$, a varifold $V \in \mathcal{V}_n(M)$ is said to be \emph{$c$-almost minimizing for $\Phi$ in $\mathscr{C}$} if there exists an annulus $A \in \mathscr{C}$ such that $V$ is $c$-almost minimizing for $\Phi$ in $A$.

  Similar, we will say that $V$ is $c$-almost minimizing for $\Phi$ in $\mathscr{C}$ by means of $\Omega_i$, if there exists an annulus $A \in \mathscr{C}$ such that $V$ is $c$-almost minimizing for $\Phi$ in $A$ by means of $\Omega_{i}$.
\end{definition}

\begin{definition}\label{def:smallannuli}
  A varifold $V \in \mathcal{V}_n(M)$ is said to be \emph{$c$-almost minimizing for $\Phi$ in small annuli} (by means of $\Omega_i$) if for every $p \in M$,  there exists $r_\text{am}(p) > 0$, such that $V$ is $c$-almost minimizing for $\Phi$ in $\an(p,s,r) \cap M$ for all $0 < s < r \leq r_\text{am}(p)$ (by means of $\Omega_i$).
\end{definition}

\begin{lemma}\label{lem:am_annuli}
  For any $L \in \mathbb{N}^+$, if $V$ is a $c$-almost minimizing varifold for $\Phi$ in every $L$-admissible collection of annuli, then $V$ is $c$-almost minimizing for $\Phi$ in small annuli.
\end{lemma}
\begin{proof}
  Suppose for the sake of contradiction that there exists $p \in M$ such that for every $r \in (0, \operatorname{inj}(M))$, there exists $s = s(r) > 0$ such that $V$ is not $c$-almost minimizing in $A(p, s, r)$.

  Therefore, we can inductively choose 
  
  \[
    \begin{aligned}
      r_1 &\in (0, \operatorname{inj}(M)), & s_1 &:= s(r_1), \\
      r_2 &\in (0, s_1/2), & s_2 &:= s(r_2),\\
       &\cdots & &\cdots \\
      r_{L} &\in (0, s_{L - 1}/2), & s_{L} &:= s(r_{L}).
    \end{aligned}
  \]
  By the construction above, $V$ is not $c$-almost minimizing in any annulus in the $L$-admissible collection of annuli
  \[
    \mathscr{C} := \{\an(p, s_1, r_1), \an(p, s_2, r_2), \cdots, \an(p, s_{L}, r_{L}\}\,.
  \]
  This yields a contradiction.
\end{proof}

In the next theorem we show that it is possible to construct a varifold which is \emph{critical}  (i.e.\ it realizes the width), satisfies a mass ratio upper bound,  is \(c\)-almost minimizing on annuli and it has first variation bounded by \(c\). Existence of a \(c\)-almost minimizing varifold can be obtained by verbatim following the arguments in~\cite[Theorem 4.10]{Pitts1981} as adapted in~\cite[Theorem 5.6]{ZhouZhu2019} to the case of Caccioppoli sets, which is relevant for us. 
The fact that the obtained varifold has also \(c\)-bounded first variation is usually obtained via a pull tight procedure. We can not perform this construction, since it is critical for us that the varifold is limit of (boundaries of) sets with uniform mass ratio bound and this property might not be  preserved by the pull tight. Instead, we will prove this via a simple cut of trick which however relies on the mass ratio bound. Note that in general an almost minimizing varifold in small annuli does not need to be stationary, as the example of a Dirac measure at a point of the Grassmannian bundle shows. Note also that it is crucial that \(n\ge 2\) since two half lines meeting at a non-planar angle define an almost minimizing varifold in small annuli with an area ratio bound, but it is not stationary.

We also need to exploit the construction in~\cite[Proposition 4.1]{DePhilippis_DeRosa2024} to obtain good bounds on the radii function in \(r_{\text{am}}(p)\). These bounds would play a key role in passing to the limit for \(c\to 0\) in the final argument of the proof of the main results.

\begin{theorem}[existence of $c$-almost minimizing varifold] \label{thm: existence of almost minimizing varifold}
  Let \(n+1\ge 3\), $(M^{n+1}, g)$ be a closed Riemannian manifold, let $\Phi$ be an elliptic integrand satisfying~\eqref{limitato}, let $\Lambda>0$ and let $c \in [0,\Lambda]$. There exists an ONVP sweepout $\Gamma $ with uniform mass ratio upper bound  $C$,  and an $n$-varifold $V \in \mathbf{C}(\Gamma)\cap \mathcal{IV}(M)$ with the following properties:
  \begin{enumerate}[\normalfont(1)]
  \item \(V\) has anisotropic first variation bounded by $c$;
  \item  \(V\) is $c$-almost minimizing for $\Phi$ in every $\bar L$-admissible collection of annuli, by means of a min-max sequence $\Omega_i$ extracted from the sweepout $\Gamma$. Here \(\bar L\) is at most \(27\) (so in particular independent of \(V\)).
 \end{enumerate}
\end{theorem}
\begin{proof}
  We begin with the ONVP sweepout \(\Gamma(t)=\partial^{*}\Omega(t)\) with uniform mass ratio upper bound $C$ constructed in Section~\ref{sec:sweepouts}.

  To show the existence of a $V\in \mathbf{C}(\Gamma)$ which satisfies conclusions \((2\)), one argues by contradiction as in the proof of~\cite[Theorem 4.10]{Pitts1981} and takes \(\bar L\) equal to the constant \(c=(3^m)^{3^m}\) defined just before Part 1 there, with  \(m=1\) (the dimension of the parameter space).  Note the construction in~\cite{Pitts1981} does not depend on \(V\) being in the critical set of a pulled tight sweepout (which is only used there to show that \(V\) is stationary). In particular, as a limit of $\Omega_j$ extracted from $\Gamma$, $V$ also has a mass ratio upper bound $C$.

  We now show that \(V\) satisfies (1). Suppose for the sake of contradiction that the varifold $V$ obtained above does not satisfy (1), and there exists a smooth vector field $X \in \mathcal{X}(M)$ with $\|X\|_{L^\infty} \leq 1$ such that
  \[
    \delta_\mathbf{\Phi}V(X) + c\int |X| d\|V\| < 0\,.
  \]
  Let $\{B_j\}^K_{j = 1}$ be a finite subcovering of $\{B_{r_\text{am}(p)}(p)\}_{p \in M}$, and let $\{\rho_j\}^K_{j=1}$ be a partition of unity subordinate to this subcovering. By linearity, it follows that
  \[
    \sum_j \left(\delta_\mathbf{\Phi}V(\rho_j X) + c \int |\rho_j X| d\|V\|\right) < 0\,,
  \]
  so at least one of the summands must be negative.
  Consequently, there exists $p \in M$ and $Y \in \mathcal{X}(M)$ with $\|Y\|_{L^\infty} \leq 1$ such that
  \[
    \delta_\mathbf{\Phi}V(Y) + c\int |Y| d\|V\| < 0\,, \quad \spt(Y) \subset B_{r_\text{am}(p)}(p)\,.
  \]
  For all sufficiently large $j \in \mathbb{N}^+$, we can find a smooth cutoff function $\eta_j: M \to [0, 1]$ such that 
  \[
    \eta_j \vert_{B_{1/j}(p)} \equiv 1\,, \quad \eta_j \vert_{(B_{2/j}(p))^c} \equiv 0\,, \quad \|\nabla \eta_j\|_\infty < 2j\,.
  \]
  By~\eqref{e:firstvaritationroughbound} and the fact that $V$ has a mass ratio upper bound $C$, we get
  \begin{multline*}
           \delta_\mathbf{\Phi}(\eta_j Y) + c\int |\eta_j Y| d\|V\|
           \\
    \leq  2 C \left(\frac{2}{j}\right)^n  \bigl(\|\nabla Y\|_{L^\infty} +\| Y\|_{L^\infty} + \|\nabla \eta_j\|_{L^\infty}\bigr) + \Lambda C \left(\frac{2}{j}\right)^n \to 0\,, \quad \text{as }j \to \infty\,.
  \end{multline*}
  In particular, we can choose $j_0 \in \mathbb{N}^+$ and define $s = 1/j_0$, $r = r_\text{am}(p)$ and $Z := (1 - \eta_{j_0}) Y$ such that
  \[
    \delta_\mathbf{\Phi}V(Z) + c\int |Z| d\|V\| < 0\,, \quad \spt(Z) \subset A(p, s, r)(p)\,.
  \]
  Let $\{\varphi_t: M \to M\}_{t \in \mathbb{R}}$ be the one-parameter family of diffeomorphisms generated by $Z$. Then there exists an interval $[0, \alpha]$ and $\varepsilon > 0$ such that for any $\Omega \in \mathcal{C}(M)$ with $\mathbf{F}(|\partial^* \Omega|, V) \leq \varepsilon$,
  \[\begin{aligned}
    &\delta_{\mathbf{\Phi}^c}\Omega(\varphi'_t) < 0\,, \quad \forall t \in [0, \alpha]\\
    &\mathbf{\Phi}^c((\varphi_{t_2})(\Omega)) > \mathbf{\Phi}^c((\varphi_{t_1})(\Omega)), \quad \forall 0 \leq t_1 < t_2 \leq \alpha\\
    & \mathbf{\Phi}^c(\Omega) - \mathbf{\Phi}^c(\varphi_\alpha(\Omega)) > \varepsilon\,.
  \end{aligned}\]
  Since $V$ is $c$-almost minimizing in $A(p, s, r)$ by means of $\Omega_i, \varepsilon_i, \delta_i$, for sufficiently large $i$, we have $\mathbf{F}(|\partial^* \Omega|, V) \leq \varepsilon$ and $\varepsilon_i < \varepsilon$. However, the inequalities above imply that
  \[
    \Omega_i \notin \mathfrak{a}^{\Phi, c}(A(p, s, r); \varepsilon_i, \delta_i)\,,
  \]
  a contradiction. Therefore, $V$ obtained above has anisotropic first variation bounded by $c$.

  To conclude, we are left to show that $V\in \mathcal{IV}(M)$. This is a direct consequence of $V$ being almost minimizing in small annuli, and can be proved as in~\cite[Lemma 4.15]{DePhilippis_DeRosa2024}, which is independent of the dimension $n$. The main tool used therein is the rectifiability theorem proved in~\cite[Theorem 1.2]{DPDRGrect}.
\end{proof}

Pitts' key idea in showing that  almost minimizing varifolds are regular  relies on the notion of \emph{replacement}.  Here, we adapt it to our scopes. In particular, the definition below is tailored to the case \(c>0\), when one expects the varifold \(V^{*}\) to have multiplicity \(1\).

\begin{definition}\label{def:replac}
 Given a varifold \(V\) and  a compact set \(K \subset M\),  we say that a varifold \(V^{*}\) is a \(\mathbf{\Phi}^{c}\)-replacement for \(V\) in \(K\)  if
\begin{itemize}
  \item[(a)] $V \res (M \setminus K) = V^* \res (M \setminus K)$;
  \item[(b)] $-c \Vol(K) \leq \|V\|(M) - \|V^*\|(M) \leq c \Vol(K)$;
\item[(c)] there exists an almost-embedded (open) hypersuface \(\Sigma\subset K\)  with
 \[
 \mathcal H^{{n-2}}(\overline{\Sigma}\setminus \Sigma)=0\]
  which is \(c\)-stable in \(\mathrm{int}(K)\) and such that   \(V\res \mathrm{Int} (K)=|\Sigma|\).
\end{itemize}

\end{definition}

The final part of the section will be dedicated to showing that a \(c\)-almost minimizing varifold  \(V\)  always has replacements and that these replacements satisfy additional properties if \(V\) possesses them.

Recall that a set of finite perimeter is locally  $\mathbf{\Phi}^c$-minimizing in an open set \(U\) if it minimizes $\mathbf{\Phi}^c$ in sufficiently small balls contained in \(U\).
\begin{lemma}[a constrained minimizing problem]\label{lem:area_ratio_am}
  Given $\varepsilon, \delta > 0$, $U \subset M$, $\Lambda > 0$, $c \in (0, \Lambda]$, and any $\Omega \in \mathfrak{a}^{\Phi, c}(U; \varepsilon, \delta)$, fix a compact subset $K \subset U$. Let $\mathcal{C}^\delta_\Omega(K)$ be the set of all $\tilde{\Omega} \in \mathcal{C}(M)$ such that there exists a sequence $\Omega = \Omega_0, \Omega_1, \ldots, \Omega_m = \tilde{\Omega}$ in $\mathcal{C}(M)$ satisfying the following conditions:
  \begin{enumerate}[\normalfont(a)]
  \item $\Omega_i \Delta  \Omega \subset K$;
  \item $\Vol(\Omega_i \Delta \Omega_{i - 1}) \leq \delta$;
  \item $\mathbf{\Phi}^c(\Omega_i) \leq \mathbf{\Phi}^c(\Omega) + \delta$, for all $i = 1, 2, \ldots, m$.
  \end{enumerate}
  Then there exists an $\Omega^* \in \mathcal{C}^\delta_\Omega(K)$, such that
  \begin{enumerate}[\normalfont(1)]
  \item $\Omega^*$ is locally $\mathbf{\Phi}^c$-minimizing in $\operatorname{int}(K)$;
  \item \(\partial^{* }\Omega^{*}\) is stable in  $\operatorname{int}(K)$;
  \item $\Omega^* \in \mathfrak{a}^{\Phi, c}(U; \varepsilon, \delta)$.
  \end{enumerate}

  Furthermore, if $\Omega$ has a mass ratio upper bound $C_0$, 
  then there exists $C= C(M, \lambda, \Lambda, C_0)$ such that $\Omega^*$ has a mass ratio upper bound $C$.
\end{lemma}

\begin{remark}
  This lemma can be regarded as an anisotropic analogue of~\cite[Lemma 5.7]{ZhouZhu2019}, but the proof differs. The key improvement is that the mass ratio bound of $\Omega^*$ is controlled by that of $\Omega$. To establish this bound, we need to forego the conclusion that $\Omega^*$ is the minimizer in $\mathcal{C}^\delta_\Omega(K)$, where we  proceed by only considering nested continuous paths.
\end{remark}

\begin{remark}\label{rmk:smooth}
By the results in~\cite{Bombieri, DPM1}, \(\Omega^*\) is (equivalent to) an open set,  its reduced boundary \(\partial^{*}\Omega^*\) is smooth except for a singular set $\Sing(\partial^{*}\Omega^*)$ with
\[
\mathcal H^{n-2}(\Sing(\partial^{*}\Omega^*))=0.
\]
\end{remark}

\begin{proof}[Proof of Lemma~\ref{lem:area_ratio_am}]
We start by defining $\mathcal{P}^\delta_{\Omega, \subset}(K)$ to be the set  of \textit{increasingly nested} continuous paths $\{\Omega_t\}_{t \in [0, 1]}$ in $\mathcal{C}(M)$ such that:
  \begin{enumerate}[(i)]
  \item $\Omega_0 = \Omega$;
  \item For all $ 0 \leq s < t \leq 1$, $\Omega_s \subset \Omega_t$, and
    \[
      \Vol(\Omega_t \Delta \Omega_s) = (t - s) \Vol(\Omega_1 \Delta \Omega_0)\,;
    \]
  \item For all $t \in [0, 1]$, $\Omega_t \Delta \Omega \subset K$;
  \item For all $ t \in [0, 1]$, $\mathbf{\Phi}^c(\Omega_t) \leq \mathbf{\Phi}^c(\Omega) + \delta$.
  \end{enumerate}
Let $\{\Omega^i_t\}^\infty_{i = 1}$ be a sequence in $\mathcal{P}^\delta_{\Omega, \subset}(K)$ such that
  \[
    \lim_{i \to \infty} \mathbf{\Phi}^c(\Omega^i_1) = \inf_{\{\Omega_t\}_t \in \mathcal{P}^\delta_{\Omega, \subset}}\{\mathbf{\Phi}^c(\Omega_1)\}\,.
  \]
  By  (ii) and (iii), for all $i$,
  \[
  \Vol(\Omega_t \Delta \Omega_s) \le (t - s) \Vol(K).
  \]
  It follows from the Arzel\`a-Ascoli theorem, up to a subsequence, $\{\Omega^i_t\}_{t \in [0, 1]}$ converges to $\{\tilde \Omega_t\}_{t \in [0,1]} \in \mathcal{P}^\delta_{\Omega, \subset}(K)$ and together with the lower semi-continuity of $\mathbf{\Phi}$, we have
  \begin{equation}\label{eqn:tilde_Omega_1}
    \mathbf{\Phi}^c(\tilde \Omega_1) = \inf_{\{\Omega_t\}_t \in \mathcal{P}^\delta_{\Omega, \subset}}\{\mathbf{\Phi}^c(\Omega_1)\}\,.
  \end{equation}

  \begin{claim}\label{claim1}
    $\tilde \Omega_1$ is \emph{locally outer-$\mathbf{\Phi}^c$-minimizing} in $\operatorname{int}(K)$ in the following sense: For each $p \in \operatorname{int} K$, there exists a ball $B_s(p) \subset \operatorname{int} K$ such that, for any $E \in \mathcal{C}(M)$, if $\tilde \Omega_1 \subset E$ and $E\Delta \tilde \Omega_1 \subset B_s(p)$, then
    \[
      \mathbf{\Phi}^c(\tilde \Omega_1) \leq \mathbf{\Phi}^c(E)\,.
    \]
  \end{claim}
Let  $r_0$ be a small number that will be fixed later, depending only on  \(M, g, \lambda, \Lambda\) and \(\delta\).  Suppose, for the sake of contradiction, that there exists $E \supset \tilde \Omega_1$ and $E\Delta \tilde \Omega_1 \subset \overline{B}_{s_0}(p_0)$ for some $B_{s_0}(p_0) \subset \operatorname{int}(K)$ and $s_0 \in (0, r_0)$ but
    \[
      \mathbf{\Phi}^c(\tilde \Omega_1) > \mathbf{\Phi}^c(E)\,.
    \]
    By a straightforward application of the direct methods of the calculus of variation, we  can choose  $E_0$ among all such sets $E$that satisfy the restrictions above and minimize $\mathbf{\Phi}^c$. Now, consider the  nested continuous path $\{\Omega'_t\}_{t \in [0, 1]}$ starting from $\tilde \Omega_1$, given by
    \[
      \Omega'_t := \left(\tilde \Omega_1 \cup B_{t \cdot s_0}(p)\right) \cap E_0\,, \quad \forall t \in [0, 1]\,.
    \]
    Note that $\Omega'_t \Delta\Omega \subset K$. By the $\mathbf{\Phi}^c$-minimizing property of $E_0$
    \[
       \mathbf{\Phi}^c(\Omega'_t) \leq \mathbf{\Phi}^c(\tilde \Omega_1 \cup B_{t \cdot s_0}(p))\,,
    \]
    and the right-hand side can be bounded by \( \mathbf{\Phi}^c(\tilde \Omega_1)+Cr_{0}^{n}\) for a constant which depends only on \(M, \lambda, \Lambda\). Hence,   since  also  $\mathbf{\Phi}^c(\tilde \Omega_1) \leq \mathbf{\Phi}^c(\Omega)$,
   \[
         \mathbf{\Phi}^c(\Omega'_t)\le \mathbf{\Phi}^c(\tilde \Omega_1)+Cr_{0}^{n}\leq \mathbf{\Phi}^c(\Omega) + \delta\,.
    \]
    provided we chose \(r_{0}\) small enough. Reparameterizing and concatenating $\tilde \Omega_t$ and $\Omega'_t$ will generate a nested continuous path in $\mathcal{P}^\delta_{\Omega, \subset}(K)$, which ends at $E_0$. This contradicts~\eqref{eqn:tilde_Omega_1}.

  \begin{claim}\label{claim:mass_ratio_upper_bound_1}
    There exists $C_1 = C_1(M, g, \Phi, \Lambda)$ such that if $\Omega$ has a mass ratio upper bound $C_0$, then $\tilde \Omega_1$ has a mass ratio upper bound $C_0 + C_1$.
  \end{claim}
  We choose $C_1$ such that for each $p \in M$ and for each $r \in (0, \operatorname{inj} M / 2)$,
    \begin{equation}\label{e:cc1}
      \mathbf{\Phi}(\partial B_r(p)) \leq \frac{C_1}{2} r^n\,, \quad \max(1, \Lambda)\Vol(B_r(p)) \leq  \frac{C_1}{2} r^n \,.
    \end{equation}
  Suppose for the sake of contradiction that there exists $p_0 \in M$ and $r_0 \in (0, \operatorname{inj} M / 2)$, such that
    \begin{equation}\label{e:cc2}
      \mathbf{\Phi}(\tilde\Omega_1; B_{r_0}(p_0)) > (C_0 + C_1) r_0^n\,.
    \end{equation}
  Since
    \begin{equation}\label{e:co}
      \mathbf{\Phi}(\Omega; B_{r_0}(p_0)) \leq C_0 r^n\,,
    \end{equation}
  and since \(\Omega\) and \(\tilde \Omega_{1}\) coincide outside \(K\),  $B_{r_0}(p_0) \cap K \neq \emptyset$.  Moreover
    \[
    \Omega \subset E_0 := (\tilde\Omega_1 \setminus B_{r_0}(p_0)) \cup \Omega \subset \tilde \Omega_1
    \]
  satisfies
    \begin{align*}
      \mathbf{\Phi}^c(E_0) &=    \mathbf{\Phi}(E_0)  - c \Vol(E_0)
      \\
      &\leq \mathbf{\Phi}( E_0 ;M \setminus \overline{B}_{r_0}(p_0)) + \mathbf{\Phi}(E_0; B_{r_0}(p_0)) +C_{1}r_0^{n}- c\Vol(\tilde \Omega_1) &&\text{by~\eqref{e:cc1}}&
      \\
       &=\mathbf{\Phi}( \tilde \Omega_{1} ;M \setminus \overline{B}_{r_0}(p_0)) + \mathbf{\Phi}( \Omega;B_{r_0}(p_0)) + C_{1}r_0^{n} - c \Vol(\tilde \Omega_1)
       \\
        &\leq \mathbf{\Phi}( \tilde \Omega_{1} ;M \setminus \overline{B}_{r_0}(p_0))  + C_0 r_0^n + C_1 r_0^n - c \Vol(\tilde \Omega_1)  &&\text{by~\eqref{e:co}}&
        \\
         &< \mathbf{\Phi}( \tilde \Omega_{1} ;M \setminus \overline{B}_{r_0}(p_0)) + \mathbf{\Phi}( \tilde \Omega_{1} ;B_{r_0}(p_0))  - c \Vol(\tilde \Omega_1)\le \mathbf{\Phi}^c(\tilde \Omega_1) &&\text{by~\eqref{e:cc2}}
    \end{align*}
  We choose $E_1$ among all $E$ satisfying $E_0 \subset E \subset \tilde \Omega_1$ that minimizing $\mathbf{\Phi}^c$ and  we define a new continuous path $\tilde \Omega'_t$ by
    \[
      \tilde \Omega'_t = \tilde \Omega_t \cap E\,,
    \]
  which satisfies $\mathbf{\Phi}^c(\tilde \Omega'_t) \leq \mathbf{\Phi}^c(\tilde \Omega_t)$. By a reparameterization of $\tilde \Omega'_t$, we obtain a new nested continuous path in $\mathcal{P}^\delta_{\Omega, \subset}(K)$. However, this path ends at $E$ and contradicts~\eqref{eqn:tilde_Omega_1}.

  Next, we consider $\mathcal{P}^\delta_{\tilde \Omega_1, \supset}(K)$ to be the subset of \textit{decreasingly nested} continuous paths $\{\Omega_t\}_{t \in [1, 2]}$ in $\mathcal{C}(M)$ such that
    \begin{enumerate}[(i)]
  \item $\Omega_1 = \tilde\Omega_{1}$;
  \item For all $ 1 \leq s < t \leq 2$, $\Omega_s \supset \Omega_t$, and
    \[
      \Vol(\Omega_t \Delta \Omega_s) = (t - s) \Vol(\Omega_1 \Delta \Omega_0)\,;
    \]
  \item For all $t \in [1, 2]$, $\Omega_t \Delta \Omega \subset K$;
  \item Fro all $ t \in [1, 2]$, $\mathbf{\Phi}^c(\Omega_t) \leq \mathbf{\Phi}^c(\Omega) + \delta$.
  \end{enumerate}
  As above,  we obtain a continuous path $\{\tilde \Omega_t\}_{t \in [1, 2]} \in \mathcal{P}^\delta_{\tilde \Omega_1, \supset}(K)$ such that
  \begin{equation}\label{eqn:tilde_Omega_2}
    \mathbf{\Phi}^c(\tilde \Omega_2) = \inf_{\{\Omega_t\}_t \in \mathcal{P}^\delta_{\tilde \Omega_1, \supset}}\{\mathbf{\Phi}^c(\Omega_2)\}\,.
  \end{equation}
  By the same arguments as above, one can prove:
  \begin{claim}
    $\tilde \Omega_2$ is \emph{locally inner-$\mathbf{\Phi}^c$-minimizing} in $\operatorname{int}(K)$ in the following sense. For each $p \in \operatorname{int} K$, there exists a ball $B_s(p) \subset \operatorname{int} K$ such that for any $E \in \mathcal{C}(M)$, if $\tilde \Omega_2 \supset E$ and $E\Delta \tilde \Omega_2 \subset B_s(p)$,  we have
    \[
      \mathbf{\Phi}^c(\tilde \Omega_2) \leq \mathbf{\Phi}^c(E)\,.
    \]
  \end{claim}

  \begin{claim}\label{claim:mass_ratio_upper_bound_2}
    There exists $C_2$ such that  $\tilde \Omega_2$ has a mass ratio upper bound $C_0 + C_1 + C_2$, where \(C_{1}\) is the constant from Claim~\ref{claim:mass_ratio_upper_bound_1}.
  \end{claim}

We now  show that $\tilde \Omega_2$ is also  \emph{locally outer-$\mathbf{\Phi}^c$-minimizing} and thus locally minimizing.

  \begin{claim}\label{claim:local_min}
    $\tilde \Omega_2$ is \emph{locally outer-$\mathbf{\Phi}^c$-minimizing} in $\operatorname{int}(K)$ in the following sense. For each $p \in \operatorname{int} K$, there exists $B_s(p) \subset \operatorname{int} K$ such that for any $E \in \mathcal{C}(M)$, if $\tilde \Omega_2 \subset E$ and $E\Delta\tilde \Omega_2 \subset B_s(p)$, then we have
    \[
      \mathbf{\Phi}^c(\tilde \Omega_2) \leq \mathbf{\Phi}^c(E)\,.
    \]
    In particular, $\tilde \Omega_2$ is locally $\mathbf{\Phi}^c$-minimizing in $\operatorname{int}(K)$.
  \end{claim}
    Note that $\tilde \Omega_1$ is locally outer-$\mathbf{\Phi}^c$-minimizing and $\tilde \Omega_2 \subset \tilde \Omega_1$. For each $p \in \operatorname{int} K$, there exists $B_s(p) \subset \operatorname{int} K$ where $\tilde \Omega_{1}$ is outer-$\mathbf{\Phi}^c$-minimizing.

Assume by  contradiction that $\tilde \Omega_2$ is not outer-$\mathbf{\Phi}^c$-minimizing in $\overline{B}_{s/2}(p)$, and  that there exists $E \in \mathcal{C}(M)$ with $\tilde \Omega_2 \subset E$ and $E\Delta \tilde \Omega_2 \subset \overline{B}_{s/2}(p)$ such that
    \[
      \mathbf{\Phi}^c(\tilde \Omega_2) > \mathbf{\Phi}^c(E)\,.
    \]
Again, we choose \(E_{0}\) which minimizes $\mathbf{\Phi}^c$ among all sets  $E$ satisfying the restrictions above. By the outer-$\mathbf{\Phi}^c$-minimizing property of $\tilde \Omega_1$, for
\[
\tilde \Omega_2 \subset E'_0 := E_0 \cap \tilde \Omega_1 \subset \tilde \Omega_1,
\]
we have
    \[
      \mathbf{\Phi}^c(E'_0) = \mathbf{\Phi}^c(E_0)\,,
    \]
  and thus $E'_0$ is also a $\mathbf{\Phi}^c$-minimizer among the sets $E$ above. Finally, we define a new continuous path $\{\tilde \Omega'_t\}_{t \in [1, 2]}$ by
    \[
      \tilde \Omega'_t = \tilde \Omega_t \cup E'_0\,,
    \]
    which satisfies $\mathbf{\Phi}^c(\tilde \Omega'_t) \leq \mathbf{\Phi}^c(\tilde \Omega_t)$ for all $t \in [1, 2]$. By a reparameterization of $\tilde \Omega'_t$, we obtain a new nested continuous path in $\mathcal{P}^\delta_{\tilde \Omega_1, \supset}(K)$. However, this path ends at $E'_0$ and contradicts~\eqref{eqn:tilde_Omega_2}.

Since \(\tilde \Omega_{2}\) is both  local outer- and local inner- minimizing, it is  indeed (locally) minimizing.
  Now, we concatenate $\{\tilde \Omega_t\}_{t \in [0, 1]}$ and $\{\tilde \Omega_t\}_{t \in [1, 2]}$  to  obtain a continuous path in $\mathcal{C}(M)$. By choosing sufficiently large $m > 0$, we see that the sequence $\tilde \Omega_0, \tilde \Omega_{2/m}, \tilde \Omega_{4/m}, \cdots, \tilde \Omega_2$ forms an admissible interpolation between \(\Omega\) and \(\tilde \Omega_{2}\). In particular, $\tilde \Omega_2 \in \mathcal{C}^\delta_\Omega(K)$ and we set  $\Omega^* := \tilde \Omega_2$. Conclusion (1) follows now immediately from Claim~\ref{claim:local_min}.

  Stability  in \(\operatorname{Int} (K)\) follows  from  the fact that   \(\Omega^{*}\) is smooth outside a small set (see Remark~\ref{rmk:smooth}) and is one-sided minimizing in \(\operatorname{Int} (K)\).

To prove conclusion (3) one argues  by contradiction as in~\cite[Lemma~5.7(iii)]{Zhou-zhu2020}. Indeed, if $\Omega^*$ is not in $\mathfrak{a}^{\Phi, c}(U; \varepsilon, \delta)$, there exists a sequence $\Omega^* = \Omega^*_0, \Omega^*_1, \ldots, \Omega^*_l$ in $\mathcal{C}(M)$ such that
  \begin{enumerate}[(i)]
  \item $\Omega^*_i \Delta \Omega^*\subset U$;
  \item $\Vol(\Omega^*_{i} \Delta\Omega^*_{i - 1}) \leq \delta$;
  \item $\mathbf{\Phi}^c(\Omega^*_i) \leq \mathbf{\Phi}^c(\Omega^*) + \delta$, for all $i = 1, 2, \ldots l$,
  \end{enumerate}
  but $\mathbf{\Phi}^c(\Omega^*_l) < \mathbf{\Phi}^c(\Omega^*) - \varepsilon$. By the construction of  $\Omega^*$, we know that
  \[
  \mathbf{\Phi}^c(\Omega^*) \leq \mathbf{\Phi}^c(\Omega).
  \]
  Therefore, the sequence $\tilde \Omega_0, \tilde \Omega_{2/m}, \tilde \Omega_{4/m}, \ldots, \tilde \Omega_2 = \Omega^*, \Omega^*_1, \ldots, \Omega^*_l$ satisfies (1)-(4) of Definition~\ref{def:am_varifolds}, but has $\mathbf{\Phi}^c(\Omega^*_l) < \mathbf{\Phi}^c(\Omega^*) - \varepsilon$, a contradiction to $\Omega \in \mathfrak{a}^{\Phi, c}(U; \varepsilon, \delta)$.

Finally, the mass ratio upper bound follows from Claim~\ref{claim:mass_ratio_upper_bound_1} and Claim~\ref{claim:mass_ratio_upper_bound_2} with $C= C_0 + C_1 + C_2$.

\end{proof}

Using the previous lemma, we can enhance~\cite[Proposition~5.8]{ZhouZhu2019} to include the desired mass ratio upper bound result.
\begin{proposition}[existence and properties of replacements]\label{prop:replacements}
  Given $\Lambda>0$, $c\in (0,\Lambda]$, $C_0 \in \mathbb{R}^+$, an open subset $U \subset M$ and a compact subset $K \subset U$, let $V \in \mathcal{V}_n(M)$ be $c$-almost minimizing for $\mathbf{\Phi}$ in the open set $U \subset M$ by means of $\Omega_i$, $\varepsilon_i$, $\delta_i$. Furthermore, suppose that every $\Omega_i$ has the mass ratio upper bound $C_0$. Then there exist $C= C(M, g, \lambda, \Lambda, C_0)$ and a \emph{$c$-replacement} $V^* \in \mathcal{V}_n(M)$ of $V$ in $K$ such that
  \begin{enumerate}[\normalfont(1)]
  \item $V^*$ is $c$-almost minimizing for $\mathbf{\Phi}$ in $U$ by means of $\Omega^*_i$, $\varepsilon_i$, $\delta_i$, for some $\Omega^*_i \in \mathcal{C}(M)$ that is $c$-stable in \(\mathrm{Int}(K)\), locally minimizes $\mathbf{\Phi}^c$ in $\operatorname{int}(K)$ and has mass ratio upper bound $C$. In particular, $V^*$ also has the same mass ratio upper bound $C$ and it satisfies the regularity property of Definition~\ref{def:replac}.
    \item If $V$ has $c$-bounded first variation in $M$, then so does $V^*$.
    \item $V, V^*\in \mathcal{IV}(M)$ and there exists $C'(M,g,\Phi,\Lambda)>0$ such that $\theta_*(V,x)\geq C'$ for any $x\in \operatorname{spt}(\|V\|)$ and $\theta_*(V^*,x)\geq C'$ for any $x\in \operatorname{spt}(\|V^*\|)$.
  \end{enumerate}
\end{proposition}
\begin{proof}\
The proof follows verbatim~\cite[Proposition~5.8]{ZhouZhu2019}, with Lemma 5.7  therein replaced by  Lemma~\ref{lem:area_ratio_am}.  Note that the regularity of  the replacement in \(\mathrm{Int}(K)\)  follows from the stability  of  \(\partial \Omega^*_i\) and Theorem~\ref{T:compact}. The integrality of $V$ and $V^*$ along with their uniform density lower bounds can be proved as in~\cite[Lemma 4.15]{DePhilippis_DeRosa2024}.
\end{proof}

\section{Regularity of min-max minimal hypersurfaces}\label{sec:Regularity of min-max minimal hypersurfaces}
In this section, we prove the regularity of  the min-max varifolds constructed in  Section~\ref{sec: Alomost minimizing varifolds and unifrom mass ratio}. We are going to use the fact that, if \(c>0\),  the regular part of a replacement will have multiplicity one, except for a set of finite \(\mathcal{H}^{n-1}\) measure where the multiplicity is two.

\begin{theorem}[regularity]
  \label{thm:regularity}
  Let \(\Lambda > 0\), let \(c\in (0,\Lambda]\) and let $V$ be the $n$-varifold constructed in Theorem~\ref{thm: existence of almost minimizing varifold}. Then $V=|\Sigma|$, where $\Sigma$ is a smooth almost embedded hypersurface. Moreover $\Theta(V,x)=1$ for every $x\in \mathcal R(\Sigma)$ and $\Theta(V,x)=2$ for every $x\in \mathcal S(\Sigma)$. Furthermore, $\Sigma$ has a mass ratio upper bound $C=C(M,g,\lambda, \Lambda)$. Also, for every $\bar L$-admissible collection of annuli, $\Sigma$ is $c$-stable in at least one annulus, where \(\bar L\) is as in Theorem~\ref{thm: existence of almost minimizing varifold}.
\end{theorem}

\begin{proof}
  Let $V$ be the $n$-varifold constructed in Theorem~\ref{thm: existence of almost minimizing varifold}. We will prove the regularity of $V$ near an arbitrary point $x\in \mbox{spt}\|V\|$.   Fix $x\in \mbox{spt}\|V\|$, and consider a radius $2\rho\leq r_\text{am}(x)$ that allows the construction of replacements as stated in Proposition~\ref{prop:replacements}.

  Consider a replacement $V'$ for $V$ in $\an(x, \rho, 2\rho)$, and let $\Sigma'$ be the $c$-stable hypersurface given by $V'$ in $\an(x, \rho,2\rho)$. Choose $t\in (\rho, 2\rho)$ such that both $\Sigma'$ and $\mathcal{S}(\Sigma')$ intersect $\partial B_t (x)$ transversally and such that
  \[
  \mathcal H^{n-3}(\operatorname{Sing}(\Sigma')\cap \partial B_t (x))=0
  \]
These properties are true for a.e. $t\in (\rho, 2\rho)$ (the latter one follows from the Eilenberg's inequality and  $\mathcal H^{n-2}(\operatorname{Sing}(\Sigma'))=0$).

  For $s<\rho$, we consider the replacement $V''$ of $V'$ in $\an(x, s, t)$, which in this annulus coincides with a smooth $c$-stable surface $\Sigma''$. We remark that $V',V''\in \mathcal{IV}(M)$ by the properties of replacements.

\medskip

\noindent\textbf{Step 1:}  We claim that for every $y\in \operatorname{Reg}(\Sigma')\cap \partial B_t (x)$, there exists a sufficiently small radius $r$, so that
      \begin{equation}\label{claim}
        \Sigma''\cap B_{t}(x)\cap B_{r}(y)=\Sigma'\cap B_{t}(x)\cap B_{r}(y).
      \end{equation}
  Given the local nature of the claim, we will assume without loss of generality that the ambient space is $\mathbb R^{n+1}$.

\medskip
\noindent
\emph{Case 1: Assume that $y\in \mathcal{R}(\Sigma')$.}  Fix $0<\varepsilon\ll 1$ to be chosen later, there exists a sufficiently small radius $r>0$, so that for every $z\in \mathcal{R}(\Sigma')\cap B_r (y)$,
  \[
  \dist (T_z\Sigma',T_y\Sigma')<\varepsilon.
  \]

  We  fix a point $z\in \mathcal{R}(\Sigma')\cap B_r (y) \setminus \an(x, s, t)$, and we consider a convex domain $\tilde C$ bounded by the union of two spherical caps with the same boundary, which is an $(n-1)$-dimensional sphere $S$ centered at $y$, with $z \in S$, with $T_zS = T_z(\Sigma' \cap \partial \tilde C)$ and such that the two caps intersects at an angle \(3\varepsilon\). In particular 
  \begin{equation}\label{utile per foliazione}
      \Sigma'\cap B_r (y)\setminus \an(x, s, t)\subset  \tilde C, 
  \end{equation}
  and the tangent cone $T_z \tilde C$ is a wedge with opening angle $3\varepsilon$.  We denote by
  \[
    C:=\Big((\tilde C-y)\setminus \frac{1}{2}(\tilde C-y)\Big)+y
    \]
  the annulus obtained by removing from $\tilde C$ a translation of $\frac 12 \tilde C$, that is concentric with $\tilde C$. Notice that $\tilde C$ and $C$ have the same tangent cones at \(z\),  $T_z \tilde C=T_z C$; see Figure~\ref{fig:case1}.

\begin{figure}[!ht]
    \centering
    \includegraphics[width=0.75\linewidth]{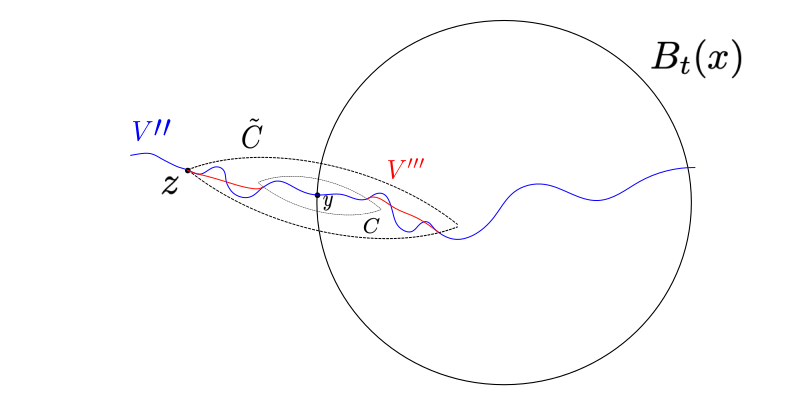}
    \caption{The construction in the proof of Case 1}
    \label{fig:case1}
\end{figure}
Let  $V'''$ be a replacement of  $V''$ inside $C$. Note that we can ensure its existence by  Proposition~\ref{prop:replacements}, since  $V''$ is $c$-almost minimizing for $\Phi$ in small annuli.  By choosing $r<r_\text{am}(y)$, there exist $0 < r_1 < r_2 \leq r$ such that $C\subset \an(y,r_1,r_2)$. Since $V''$ is $c$-almost minimizing in  $\an(y,r_1,r_2)$, $V''$ is $c$-almost minimizing in $C$ and we can apply Proposition~\ref{prop:replacements} there.

  We claim that $V'''$ is regular at $z$. Since $V'''$ coincides with $V''$ outside $C$, we know that the family $\{(\eta_{z,r'})_\# V'''\}_{r'<r}$ has uniformly bounded mass and consequently $TV(z, V''')\neq \emptyset$.

  Up to rotation,  we can assume that
  \[
  T_zC=\{|p_{n+1}|\leq \tan(3\varepsilon) p_1\}\qquad \textrm{and}\qquad T_zS=\mbox{span}(e_2,\dots,e_{n})=:\ell,
  \]
  where we have used the notation  $p=(p_1,\dots,p_{n+1})$. Since $z\in \mathcal{R}(\Sigma')$, we deduce that  every $W \in TV(z, V''')$ coincides with an half plane on one side of \(\ell\):
  \begin{equation}\label{plane}
    W\res \{p_1 \leq 0\}=|T_z\Sigma'|\res \{p_1 \leq 0\}.
  \end{equation}

  To proceed, we will need the following two lemmas, whose proofs are postponed to the end of the section.

  \begin{lemma}\label{lemma1}
    Denoting $H:=\{p_1>0\}$, there exist \(\nu_1, \nu_2\in \mathbb S^{n}\)  (possibly $\nu_1=\nu_2$), and $W\in TV(z,V''')$ such that
    \begin{equation}\label{c22}
      \operatorname{spt}(W\res \{p_1 \geq 0\}) \subset \overline{H}\cap \{\langle \nu_1 , p\rangle \geq 0 \}\cap \{\langle \nu_2 , p\rangle \leq 0 \}\subset T_zC,
    \end{equation}
    and one of the following properties holds:
    \begin{equation}\label{c1}
      |\{\langle \nu_1 , p\rangle = 0 \}|\res \overline{H} + |\{\langle \nu_2 , p\rangle = 0 \}| \res \overline{H}  \leq W\res \overline{H},\qquad \mbox{if $\nu_1\neq \nu_2$}
    \end{equation}
    \begin{equation}\label{cc1}
      |\{\langle \nu_1 , p\rangle = 0 \}|\res \overline{H}  = W\res \overline{H},\qquad \mbox{if $\nu_1= \nu_2$.}
    \end{equation}
  \end{lemma}

  \begin{lemma}\label{lem:unosolo}
    There exists  \(\varepsilon=\varepsilon(\lambda)\)  sufficiently small such that the  varifold $W\in TV(z,V''')$ constructed in Lemma~\ref{lemma1} satisfies $W=|T_z\Sigma'|$.
  \end{lemma}

By combining the two steps above, we deduce that at least one blowup of $V'''$ at the point $z$ is a hyperplane with multiplicity $1$. Now the same graphicality argument as in~\cite[Proof of Proposition 4.14, Step 1, Case 2]{DePhilippis_DeRosa2024} allows to conclude that $V'$ and $V'''$ glue smoothly at $z$. In particular, \(z\) is a multiplicity one point for \(V'''\), and this is true also in a neighborhood.  By  unique continuation, there exists a connected component of $V'''\res C$ which coincides with $\Sigma'\cap C$.

Note that the $\Phi$-anisotropic mean curvature $h_\Phi^{\partial \tilde C \setminus S}$ of the smooth part $\partial \tilde C \setminus S$ is bounded from below by
 \[
 h_\Phi^{\partial \tilde C \setminus S}\geq \frac{C(\varepsilon,n,\Phi)}r
 \]
  where $C(\varepsilon,n,\Phi)>0$ depends just on the prescribed $\varepsilon,n,\Phi$. By choosing  $r$ small enough, we can ensure  that $c<h_\Phi^{\partial \tilde C \setminus S}$. Hence, we can apply the maximum principle~\cite{DDH} to deduce that
  \begin{equation}\label{applicazioneprincipiomassimo}
    \Sigma'\cap \partial \tilde  C\subset \spt(V''') \cap \partial \tilde C \subset \mbox{spt}(V'''\res (\mathbb R^{n+1}\setminus \tilde C)),
  \end{equation}
  otherwise, we would have that the smooth $c$-stable hypersurface $\Sigma''':=V'''\res \tilde C$ touches $\partial \tilde C$ from the inside, which contradicts the maximum principle.  Since
  \[
  \spt(V'''\res (\mathbb R^{n+1}\setminus \tilde C))\subset \spt(V'''\res (\mathbb R^{n+1}\setminus C))= \spt(V''\res (\mathbb R^{n+1}\setminus C)),
  \]
  we deduce from~\eqref{applicazioneprincipiomassimo} that
  \[
    \Sigma'\cap \partial \tilde  C \subset \mbox{spt}(V'')\,.
  \]
  
  By the arbitrariness of the point $z\in \mathcal{R}(\Sigma')\cap B_r (y) \setminus \an(x, s, t)$, we can repeat the same argument above with a continuous $1$-parameter family $\{z_\alpha\}_{\alpha \in [0,1]}\subset \mathcal{R}(\Sigma')\cap B_r (y) \setminus \an(x, s, t)$, with $z_1=z$, $z_0=y$, and $d(z_\alpha,y)$ increasing in $\alpha$, so that by continuity and the fact that $\tilde C_0=\{y\}$, the associated $1$-parameter family $\tilde C_\alpha$ satisfies:\begin{equation}\label{foliation}
    \tilde C\subset\bigcup_{\alpha \in [0,1]}\partial \tilde  C_\alpha.
  \end{equation}
  Arguing as above for every $\alpha \in [0,1]$, we deduce that
  \[
    \Sigma'\cap \tilde C\overset{\eqref{foliation}}{\subset} \bigcup_{\alpha \in [0,1]}(\Sigma'\cap \partial \tilde  C_\alpha) \subset \mbox{spt}(V'').
    \]
  Since $V''\in \mathcal{IV}(M)$ by Proposition~\ref{prop:replacements}, we deduce that $|\Sigma'\cap \tilde C|\leq V''$. Hence, for every $Z\in TV(y,V'')$ we have $|T_y \Sigma'|\leq Z$. Since stationarity is preserved in the blow up:  $\delta_\mathbf{\Phi} Z=\delta_\mathbf{\Phi} |T_y\Sigma'|=0$ where we are using the same notation \(\Phi\) for the frozen functional \(\Phi(v)=\Phi(x,v)\). By the linearity of the $\Phi$-anisotropic first variation, we deduce that
  \[
  \delta_\mathbf{\Phi}(Z-|T_y\Sigma'|)=0,
  \]
   where the difference is always to be intended between measures in the varifold sense.
   Moreover, since $V' \res (M \setminus \an(x, s, t)) = V'' \res (M \setminus \an(x, s, t))$, we have that (up to a rotation)
   \[
     \spt(Z-|T_y\Sigma'|)\subset \{p_1\geq 0\}, \mbox{ and } T_y\Sigma'\cap \{p_1= 0\}= \{p_1= 0, p_{n+1}=0\}.
     \]
  Again arguing as in the proof of~\cite[Lemma 6.2]{DePhilippis_DeRosa2024}, we deduce the existence of $L=L(n,\lambda,c)$ depending only on the dimension $n$, the integrand $\Phi$ and the anisotropic constant mean curvature $c$, such that
  \[
  \sup_{\spt(Z-T_y\Sigma')}\,\frac{ |\langle x, e_{n+1}\rangle|}{\langle x , e_1\rangle}\le L.
  \]
  By the maximum principle~\cite{DDH}, this is possible only if $Z=|T_y\Sigma'|$. Again the same graphicality argument as in~\cite[Proof of Proposition 4.14, Step 1, Case 2]{DePhilippis_DeRosa2024} allows to conclude the desired~\eqref{claim}.

\medskip
\noindent
  {\emph{Case 2: Assume that $y\in \mathcal{S}(\Sigma')$.} } As observed in Section~\ref{sec:cmc}, by the proof of~\cite[Lemma 5.1]{DePhilippis_DeRosa2024}, there exists a sufficiently small radius $r$, so that $ \mathcal S(\Sigma')\cap B_r (y)\cap \partial B_t (x)$ is an $(n-2)$-dimensional $C^1$  graph in $\Sigma'\cap \partial B_t (x)$. Since $y\in \mathcal S(\Sigma')$, and we have the mass ratio upper bound of Proposition~\ref{prop:replacements}, the family $\{(\eta_{y,r'})_\# V''\}_{r'<r}$ has uniformly bounded mass (from above and below) and consequently $TV(y, V'')\neq \emptyset$.
  Up to rotation, denoting $p=(p_1,\dots,p_{n+1})$, we can assume that $T_y(\partial B_t(x))=e_1^\perp$, that $\eta_{y,r'}(B_t(x))\subset \{p_1 \geq 0\}$ and that $T_y(\mathcal{S}(\Sigma') \cap \partial B_t(x))=\text{span}( e_2, \dots, e_{n-1})=\ell$.

  Since $y\in \mathcal S(\Sigma')$, we deduce that for every $Z \in TV(y, V'')$,
  \begin{equation}\label{plane2}
    Z\res \{p_1 \leq 0\}=2|T_y\Sigma'|\res \{p_1 \leq 0\}, \qquad \mbox{where $T_y\Sigma' \neq e_1^\perp$}.
  \end{equation}
  Fix $Z \in TV(y, V'')$ and denote by $\{r_j\}$ the sequence of radii such that $W_j:=(\eta_{y,r_j})_\# V''$ converges to $Z$.

 Moreover, let $\gamma = T_y(\mathcal{S}(\Sigma'))$. In the following, we denote with $\mathbb{B}_s$ the Euclidean ball centered at $0$ of radius $s$ and with $U_s(\gamma)$ the Euclidean $s$-tubular neighborhood of $\gamma$. For every $\alpha \in (0,1)$, by Case~1 above, we know that there exists $N(\alpha)\in \mathbb N$ such that for every $j\geq N$, it holds 
 \[
    W_j\res (\mathbb{B}_{1/\alpha} \cap U_{1/\alpha}(\gamma)\setminus U_\alpha(\gamma)) \geq |\eta_{y,r_j}(\Sigma')| \res (\mathbb{B}_{1/\alpha} \cap U_{1/\alpha}(\gamma)\setminus U_\alpha(\gamma))
\] in the sense of varifolds, and $W_j$ is a sequence of $(cr_j)$-stable almost embedded smooth hypersurfaces in $(\mathbb{B}_{1/\alpha} \cap U_{1/\alpha}(\gamma)\setminus U_\alpha(\gamma))$, outside of $\mathcal H^{n-2}$-measure zero singular sets.

  Since $r_j \downarrow 0$, by Theorem~\ref{compactness}, we deduce that $W_j\res (\mathbb{B}_{1/\alpha} \cap U_{1/\alpha}(\gamma)\setminus U_\alpha(\gamma))$ converge smoothly (with integer multiplicity) to $\Sigma^\alpha$, where $\Sigma^\alpha$ is a stable embedded hypersurface with $\mathcal H^{n-2}(\operatorname{Sing}(\Sigma^\alpha))=0$. Since $Z\res (\mathbb{B}_{1/\alpha} \cap U_{1/\alpha}(\gamma)\setminus U_\alpha(\gamma))=\Sigma^\alpha$, by~\eqref{plane2} we obtain the following inequality in the sense of varifolds
  \[
    Z\res (\mathbb{B}_{1/\alpha} \cap U_{1/\alpha}(\gamma)\setminus U_\alpha(\gamma)) \geq 2|T_y\Sigma'| \res (\mathbb{B}_{1/\alpha} \cap U_{1/\alpha}(\gamma)\setminus U_\alpha(\gamma)), \qquad \forall \alpha\in (0,1),
    \]
  and consequently,
  \[
    Z\res \gamma^c \geq 2|T_y\Sigma'| \res  \gamma^c.
    \]
  By the uniform upper density estimates obtained in Proposition~\ref{prop:replacements}, we deduce that
  \begin{equation}\label{tangent plane}
    Z\geq 2|T_y\Sigma'|.
  \end{equation}
  Since both $Z$ and $2|T_y\Sigma'|$ are stationary, we deduce that $Z':=Z-2|T_y\Sigma'| \subset \{p_1 \geq 0\}$ is stationary, where again the difference is to be intended in the space of varifolds.

  With the same argument used in the proof of~\cite[Proposition 5.3]{DePhilippis_DeRosa2024} to obtain~\cite[Equations (5.5)-(5.6)]{DePhilippis_DeRosa2024}, we can prove that $Z'$ is contained in a wedge $L:=\{|p_{n+1}|\leq a p_1, \, p_1 \geq 0\}$ for some $a>0$. We claim that $Z'=0$, and consequently that $Z=2|T_{y} \Sigma'|$.
  Indeed if by contradiction $Z'\neq 0$, there exists $\bar h:=\min\{h\geq 0: \{p_1 = h\}\cap \text{spt}(Z')\neq 0\}$. By the maximum principle, we deduce that $\{p_1 = \bar h\}\subset \text{spt}(Z')$. But this cannot be true as $\{p_1 = \bar h\}$ is not entirely contained in the wedge $L$. This is the desired contradiction. In conclusion, we have proved that
  \[
    TV (y, V'')=\{2|T_{y} \Sigma'|\}.
    \]
The same graphicality argument as in~\cite[Proof of Proposition 4.14, Step 1, Case 2]{DePhilippis_DeRosa2024} allows to obtain the desired~\eqref{claim}.

\medskip
  \noindent \textbf{Step 2:}  We claim that $V\res (B_{2\rho}(x)\setminus \{x\})=|\Sigma| \res (B_{2\rho}(x)\setminus \{x\})$, where $\Sigma$ is a smooth almost embedded $c$-stable hypersurface in $B_{2\rho}(x)\setminus \{x\}$, except for a set of \(\mathcal H^{n-2}\) measure \(0\).

  This can be proven exactly as in the proof of~\cite[Proposition 4.14, Step 2]{DePhilippis_DeRosa2024}, which is independent on the dimension $n$.

\medskip

 \noindent \textbf{Step 3}: We claim that $V\res B_{2\rho}(x)=|\Sigma| \res B_{2\rho}(x)$, where $\Sigma$ is a smooth almost embedded $c$-stable hypersurface in $B_{2\rho}(x)$, except for a set of \(\mathcal H^{n-2}\) measure \(0\).

Note that by the compactness of \(M\) and the previous steps, the conclusion is true except for finitely many points (the centers of the annuli).  If $n+1 \geq 4$, there is nothing to prove since we have possibly only  added a finite set of points  to the singular set, which still has vanishing \(\mathcal H^{n-2}\) measure.

  In the case $n+1=3$, we know from Step 2 that $V$ is smooth in $B_{2\rho}(x)\setminus \{x\}$ and we need to show that $x\in \operatorname{Reg}(\Sigma)$. By the previous step, $\Sigma$ is a smooth almost embedded $c$-stable surface in $B_{2\rho}(x)\setminus \{x\}$. To  remove the singularity we aim  to apply\footnote{Although~\cite[Theorem 2, Page 250]{Whitesingrimovibili} is stated for embedded surfaces, the proof requires minor adaptations to work for the almost embedded surface $\Sigma$.}~\cite[Theorem 2, Page 250]{Whitesingrimovibili}, provided we can show that
  \begin{equation}\label{curv}
    \int_{\Sigma\cap B_{\rho}(x)}  |A|^2 <\infty.
\end{equation}
This follows from the inequality~\eqref{eq:stability}, the mass ratio bound and a classical capacity argument. To verify it, note that using the bound
\[
\mathcal H^{2}(\Sigma \cap B_{r}(x))\le C r^{2}
\]
and a standard logarithm cut-off trick, one can easily construct a sequence of functions \(\psi_{\varepsilon}\in C_{c}^1(B_{2\rho}(x)\setminus\{x\})\) such that, as $\varepsilon\downarrow 0$,
\[
\psi_{\varepsilon}(y) \uparrow 1\qquad \text{for all \(y\in B_{\rho}(x)\setminus\{x\}\)}
\]
and
\[
\sup_{\varepsilon} \int_\Sigma|\nabla \psi_{\varepsilon}|^2+|\psi_{\varepsilon}|^2 <+\infty.
\]
By plugging this sequence in~\eqref{eq:stability} and by letting \(\varepsilon\to 0\), we get~\eqref{curv}.

\medskip
 \noindent \textbf{Step 4:} {\emph{We prove that for every $\bar L$-admissible collection of annuli, $\Sigma$ is $c$-stable in at least one annulus.}}

  One can argue as in~\cite[Theorem 3.3]{Pitts1981}, except changing the target functional from the area functional to $\mathbf{\Phi}^c$, that if a varifold is $c$-almost minimizing for $\Phi$ in an open set, then the varifold is $c$-stable in the same open set. Then the conclusion follows from Theorem~\ref{thm: existence of almost minimizing varifold} (2).

\end{proof}

 \begin{proof}[Proof of Lemma~\ref{lemma1}] The proof follows from~\cite[Lemma 2.11]{DPM1}, which in turn is inspired by~\cite{Hardt77}. The idea is to show that there exists a tangent varifold \(W\) which is contained in a ``minimal'' wedge. If its support does not contain the boundaries of the wedge then one can construct a nonaffine graph which is \(\Phi_{0}\) stationary and that, by Hopf maximum principle, has  a smaller slope at the origin. By taking a further blow up, we get a contradiction with the minimality of the wedge.

 Recall that \(H=\{p_{1}\ge 0\}\). We will   denote with $G^+:=G\cap H$ for every $G\subset\mathbb R^{n+1}$. Arguing as in the proof of~\cite[Lemma 6.2]{DePhilippis_DeRosa2024}, we deduce the existence of $L=L(n,\Phi,c)$ depending only on the integrand $\Phi$ and the constant mean curvature $c$ such that
    \begin{equation}\label{c2}
      \sup_{(\spt  W)^{+}}\,\frac{|\langle p, e_{n+1}\rangle|}{\langle p , e_1\rangle}\le L\,,\qquad\forall\, W\in TV(z,V'')\,.
    \end{equation}
    Moreover
    \begin{equation}
      \label{c5}
      \spt W \cap \partial H = \ell=\mbox{span}(e_2,\dots,e_{n})\,,\qquad\forall\, W\in TV(z,V'')\,.
    \end{equation}
    We define $\xi:TV(z,V'')\to[-L,L]$ as
    \begin{equation}\label{c6}
      \xi(W)=\inf_{(\spt W)^{+}}\,\frac {\langle p, e_{n+1}\rangle}{\langle p , e_1\rangle}\,,\qquad \qquad\forall\, W\in TV(z,V'').
    \end{equation}
    As in the proof of~\cite[Lemma 5.4]{DPM1}, one can easily check that $\xi$ is upper semicontinuous on $TV(z,\Omega)$ with respect to the $L^1_{{\rm loc}}(\mathbb R^{n+1})$ convergence and hence the existence of $W_{1}\in TV(z,V'')$ such that
    \begin{equation}\label{c7}
      \xi(W_1)\ge\xi(W)\,,\qquad\forall W\in TV(z,V'')\,.
    \end{equation}

    Let us fix \(\alpha \in (-\pi/2,\pi /2)\) so that $\tan \alpha=\xi(W_1)$ and set
    \[
      \nu_1=\cos \alpha\, e_{n+1}-\sin \alpha \,e_1\in \mathbb S^{n}\,,\qquad H_1=\Big\{p\in H:\langle p, \nu_1\rangle \ge0\Big\}\,.
    \]
    We claim  that
    \begin{equation}\label{ccccc}
      (\partial H_1)^+ \subset (\spt W_1)^+\,.
    \end{equation}
    Indeed, by definition of $\xi$, it holds
    \begin{equation}
      \label{c8}
    (\spt W_1)^+\subset  H_1\,.
    \end{equation}
    Moreover, denoting with $w:\{q\in\mathbb R^{n}:q_1>0\}\to [-\infty,+\infty)$ the function satisfying
    \[
      w(q)=\inf\Big\{t\in\mathbb R:  (q,t)\in \spt W_1\Big\}\,,\qquad \forall q\in \mathbb R^{n}:q_1>0\,,
    \]
    we deduce from~\eqref{c5},~\eqref{c8}, and the lower semicontinuity of $w$ that
    \begin{align}
      \label{c9}
      (\spt W_{1})^{+}\subset\Big\{p\in H:p_{n+1}\ge w(p_1,\dots,p_n)\Big\}\,,&&
      \\
      \label{c10}
      w(q)\ge \xi(W_1)\,q_1\,,&&\quad\forall q\in\mathbb R^{n}:q_1>0\,.
    \end{align}
    If~\eqref{ccccc} fails, then there exists $\bar p:=(\bar q,\bar p_{n+1})=:(\bar p_1,\dots,\bar p_n, \bar p_{n+1})\in (\spt W_{1})^{+}$ such that
    \begin{equation}
      \label{c11}
      w(\bar q)> \xi(W_1)\,\bar p_1\,.
    \end{equation}
    By~\eqref{c10} and~\eqref{c11}, if we set \(\bar r=|\bar q|\), and $D_{\bar r}=B_{\bar r}\cap \mbox{span}(e_1,\dots,e_{n})$, then we can find \(\varphi\in C^{1,1}(\partial(D_{\bar r}^+))\) such that
    \begin{eqnarray}\label{c12}
      w(q)\ge \varphi (q)\ge \xi(W_1)\,\langle q, e_1\rangle\,,&&\qquad\forall q\in \partial (D_{\bar r}\cap H)
      \\
      \varphi(\bar q)>\xi(W_1)\, \langle \bar q, e_1\rangle=\xi(W_1)\,  \bar p_1\,.&&
    \end{eqnarray}
    In particular, $\varphi=0 $ on $D_{\bar r}\cap\partial H$. By part two of~\cite[Lemma 2.11]{DPM1}, there exists \(u\in C^{1,1}(D_{\bar r}^+)\cap \mbox{Lip}(\overline D_{\bar r}^+)\) such that, if we set $G_0^\#(q)=\Phi(0;q,-1)$ for $q\in\mathbb R^{n}$, then
    \begin{equation*}
      \begin{cases}
        \mbox{div} (\nabla_{q} G_0^\#(\nabla u))=0\,,\qquad&\textrm{in \(D_{\bar r}^+\)}\,,
        \\
        u=\varphi\,,&\textrm{on \(\partial (D_{\bar r}^+)\)}\,,
      \end{cases}
    \end{equation*}
    with
    \begin{equation}\label{c14}
      |\nabla u(0)|=|\langle \nabla u(0), e_1\rangle |, \qquad \mbox{and} \qquad \langle \nabla u(0), e_1\rangle >\xi(W_1)\,.
    \end{equation}
  Recalling that  $\delta_\mathbf{\Phi_0}W=0$, where $\Phi_0(v)= \Phi(z,v)$ we can combine~\eqref{c9} and~\eqref{c12} with the maximum principle~\cite{DDH}, to deduce that
    \begin{equation}
      \label{c15}
 (\spt W_{1})^{+}  \cap(D_{\bar r}^+\times\mathbb R)  \subset  \Big\{(q,t)\in D_{\bar r}^+\times\mathbb R:t\le u(q)\Big\} \,.
    \end{equation}
    We now pick a sequence $\{s_h\}$ such that $s_h\to 0$ as $h\to\infty$ and $\eta_{0,s_h}(W_{1})\to\widetilde{W}\in T(0,W_{1})$. By~\eqref{c15} and $u(0)=0$, we get
    \[
      (\spt \widetilde W)^{+}\subset   \Big\{(q,t):t\ge \langle \nabla u_0(0), e_1\rangle \langle q , e_1\rangle \Big\}\,,
    \]
    so that, thanks to~\eqref{c14}, $\xi(\widetilde{W})>\xi(W_1)$. Since $\widetilde{W}\in TV(0,W_1)\subset TV(z,W)$, this contradicts~\eqref{c7}, thereby completing the proof of~\eqref{ccccc} and identifying $\nu_1$ in the statement~\eqref{c1} of the lemma.

    Since $TV(0,W_1)\subset TV(z,W)$, in order to identify $\nu_2$ we can argue analogously as above performing another blow-up  $W_1$ to obtain  $W_2\in TV(0,W_{1})$  and to identify the vector $\nu_2\in \mathbb S^n$. Since \(W_{2} \subset TV(0,W_1)\subset TV(z,V'')\), this concludes the proof.
  \end{proof}

    \begin{proof}[Proof of Lemma~\ref{lem:unosolo} ]
    Recall that  $\Phi_0(v)=\Phi(z,v)$ is the blow-up integrand. We recall that, by~\cite[Equation (10)]{DDH}, the \(\Phi_{0}\)-anisotropic first variation of an half-plane with normal $\nu$ and conormal $\eta$, and bounded by an $(n-1)$-plane $ \ell$ is given by
    \begin{equation}\label{fvhp}
     (\Phi_0(\nu)\eta-\langle D \Phi_0(\nu), \eta\rangle \nu)  \mathcal H^{n-1}\res  \ell.
    \end{equation}
    Assume by contradiction that $\nu_1\neq\nu_2$. Combining~\eqref{plane} with~\eqref{c1}, we have that:
    \[
    W \geq |T_z\Sigma'|\res  \{p_1 \leq 0\}+ |\{\langle \nu_1 , p\rangle = 0 \}|\res \overline{H} +|\{\langle \nu_2 , p\rangle = 0 \}| \res \overline{H}  =:\tilde W.
    \]
    By~\eqref{fvhp}, we easily compute that, up to choose $\varepsilon$ small enough (depending only on $\lambda$):
    \[
      \delta_\mathbf{\Phi_0} \tilde W:= (- \Phi_0(e_{n+1})e_1+\langle  D\Phi_0(e_{n+1}), e_1\rangle e_{n+1}+w)\mathcal H^{n-1}\res \ell ,
      \]
    where $w\in \operatorname{span}(e_1,e_{n+1})\subset \mathbb R^{n+1}$ with $|w|\leq O(\varepsilon)\ll1$. Since $\delta_\mathbf{\Phi_0} W=0$, by the linearity of the anisotropic first variation we deduce that
    \[
      \delta_\mathbf{\Phi_0}(W-\tilde W)=( \Phi_0(e_{n+1})e_1-\langle D\Phi_0(e_{n+1}), e_1\rangle e_{n+1}-w)\mathcal H^{n-1}\res \ell.
      \]
    We now observe that there exists $\bar \nu \in \operatorname{span}(e_1,e_{n+1})$, with $\bar \nu_1>0$, such that
    \begin{equation}\label{brutta}
        \frac{\Phi_0(e_{n+1})e_1-\langle D \Phi_0(e_{n+1}), e_1\rangle e_{n+1}-w}{|\Phi_0(e_{n+1})e_1-\langle D \Phi_0(e_{n+1}), e_1\rangle e_{n+1}-w|} =\frac{\Phi_0(\bar \nu) \tilde \nu-\langle D \Phi_0(\bar \nu), \tilde \nu\rangle \bar \nu}{|\Phi_0(\bar \nu)\tilde \nu-\langle D \Phi_0(\bar \nu), \tilde \nu\rangle \bar \nu|},
    \end{equation}
    where $\tilde \nu:=(-\bar\nu_{n+1},0,\dots,0,\bar\nu_{1})\in \operatorname{span}(e_1,e_{n+1})$ is orthogonal to $\bar \nu$. This is easily obtained by continuity of the right hand side in~\eqref{brutta} and the intermediate value theorem, provided $\varepsilon$ is chosen small enough.  Hence, there exists $\theta_3\geq 0$ such that:
    \[
      \delta_\mathbf{\Phi_0}(\theta_3|\bar \nu^\perp|\res\{p_1\geq 0\})=(-\Phi_0(e_{n+1})e_1+\langle D \Phi_0(e_{n+1}), e_1\rangle e_{n+1}+w)\mathcal H^{n-1}\res \ell,
      \]
    where we have denoted by \(|\bar\nu^{\perp}|\) the varifold associated with the plane  perpendicular to \(\bar\nu\). Again by linearity of the  first variation, we conclude that
    \[
      \delta_\mathbf{\Phi_0}(W-\tilde W +\theta_3|\bar \nu^\perp|\res\{p_1\geq 0\})=0,
      \]
    which contradicts the maximum principle~\cite{DDH}, since spt$(W-\tilde W +\theta_3|\bar \nu^\perp|\res\{p_1\geq 0\})$ is contained in a wedge up to choose $\varepsilon$ small enough.

    Hence $\nu_1=\nu_2$ and~\eqref{cc1} holds. In particular
    \[
      W=|T_z\Sigma'| \res (\mathbb R^n \setminus H) +|\{p\in H: \langle \nu_1 , p\rangle = 0 \}|
      \]
    and, since \(\delta_{\Phi_{0}}W=0\), it is easy to see that the only possibility is that \(\nu_{1}\perp T_{z}\Sigma'\), concluding the proof.
  \end{proof}

\section{Proofs of the main results}\label{sec:mainproofs}
The starting point is the following theorem, which is stronger than the statement of Theorem~\ref{thm:main2}:
\begin{theorem}\label{thm:main3}
  Given \(\lambda, \Lambda>0\), any smooth closed Riemannian manifold $(M^{n+1}, g)$, any smooth elliptic integrand $\Phi$ satisfying~\eqref{limitato}, and $ c\in (0,\Lambda]$, there exists an almost embedded hypersurface $\Sigma$ with anisotropic mean curvature equal to \(c\), such that $\mathcal{H}^{n-2}(\operatorname{Sing}(\Sigma)) = 0$. Moreover, $\Sigma$ has a mass ratio upper bound given by  a constant $C= C(M, g, \lambda, \Lambda)$, with $\Theta(V,x)=1$ for every $x\in \mathcal R(\Sigma)$ and $\Theta(V,x)=2$ for every $x\in \mathcal S(\Sigma)$. Additionally,
  \[
  \mathbf{\Phi}(\Sigma)\leq 2(W^c_\Phi(M, g)+c\Vol(M)).
  \]
  Furthermore, for every $\bar L$-admissible collection of annuli, $\Sigma$ is $c$-stable in at least one annulus, where \(\bar L\) is as in Theorem~\ref{thm: existence of almost minimizing varifold}.
\end{theorem}
\begin{proof}
This immediately follows by combining Theorem~\ref{thm: existence of almost minimizing varifold} with Theorem~\ref{thm:regularity}.
\end{proof}

\begin{proof}[Proof of Theorem~\ref{thm:main2}]
  Theorem~\ref{thm:main2} is implied by Theorem~\ref{thm:main3}.
\end{proof}

\begin{proof}[Proof of Theorem~\ref{thm:main1}]
  Consider the sequence $c_k\to 0$. By Theorem~\ref{thm:main3},  there exists  a sequence of nontrivial hypersurfaces $\Sigma_k$ , which are smooth and almost embedded outside of a singular set of zero $\mathcal H^{n-2}$-measure, with constant anisotropic mean curvature $c_k$ and mass ratio upper bound $C$ independent of $k$, such that
  \[
  \mathbf{\Phi}(\Sigma_k)\leq 2(W^{c_k}_\Phi(M, g)+c_k\Vol(M)).
  \]
  In particular, by~\eqref{limitato} and Lemma~\ref{lemma:utiledopo}, we get that
  \[
  \sup_{k} \mathcal{H}^n(\Sigma_k)< \infty.
  \]
  Moreover, for every $\bar L$-admissible collection of annuli, $\Sigma_k$ is $c_k$-stable in at least one annulus.

  Therefore, $|\Sigma_k|$ converges to a \(\Phi\)-stationary varifold $V$. It suffices to show that $V$ is associated with a smooth hypersurface (with integer multiplicities) except for a codimension $2$ Hausdorff measure $0$ set.

  We claim that for every $p \in M$, there exists a radius $r_s(p) > 0$, such that for any $0 < s < r \leq r_s(p)$, there exists a subsequence $\{\Sigma_{k_l}\}$ such that every $\Sigma_{k_l}$ is $c_{k_l}$-stable in $A(p, s, r)$. Indeed, if this is true, by Theorem~\ref{T:compact} (ii), $\spt V$ is smooth and stable in $B_{r_s(p)}(p) \setminus \{p\}$ except for a codimension $2$ Hausdorff measure $0$ set. If $n \geq 3$, we are done by a finite covering argument; if $n = 2$, using the mass ratio uniform upper bound and arguing as in the proof of Theorem~\ref{thm:regularity} (Step 3), we can remove the singular point $p$ by means of the stability inequality and conclude the proof.

  Suppose for the sake of contradiction that the claim fails. Then there exists $p \in M$ such that for every $r \in (0, \operatorname{inj}(M))$, there exists an $s(r) > 0$ and $N(r) \in \mathbb{N}^+$ such that for every $k \geq N(r)$, $\Sigma_k$ is not $c_k$-stable in $A(p, s(r), r)$. Therefore, we can inductively choose 
  \[
    \begin{aligned}
      r_1 &\in (0, \operatorname{inj}(M)), & s_1 &:= s(r_1), & N_1 &:= N(r_1), \\
      r_2 &\in (0, s_1/2), & s_2 &:= s(r_2), & N_2 &:= N(r_2),\\
      & \cdots && \cdots && \cdots\\
      r_{\bar L} &\in (0, s_{\bar L - 1}/2), & s_{\bar L} &:= s(r_{\bar L}, & N_{\bar L} &:= N(r_{\bar L}).
    \end{aligned}
  \]
  Then, let $N := \max\{N_1, N_2, \cdots, N_{\bar L}\}$, and we see that $\Sigma_N$ is not $c_N$-stable in any annulus in the collection
  \[
    \mathscr{C} := \{\an(p, s_1, r_1), \an(p, s_2, r_2), \cdots, \an(p, s_{\bar L}, r_{\bar L})\}\,.
  \]
  However, by  construction, $\mathscr{C}$ is an $\bar L$-admissible collection of annuli, contradicting that $\Sigma_N$ is $c_N$-stable in at least one annulus, for every $\bar L$-admissible collection of annuli. This completes our proof.

\end{proof}

\appendix

\section{Compactness} \label{sec:compactness}
In this section, we state the main compactness criterion for stable hypersurfaces with bounded anisotropic mean curvature. Being the theory local, and since a Riemannian metric can be absorbed into the anisotropy, we can assume that the ambient space is \(\mathbb R^{n+1}\) with the euclidean metric.
\begin{theorem}\label{compactness}
  Consider $C \geq 0$, and a sequence of $c_k$-stable hypersurfaces $\Sigma_k$ in $B_{4r} \subset M$ with respect to a uniformly elliptic anisotropic integrand $\Phi$, with singular sets satisfying $\mathcal{H}^{n-2}(\operatorname{Sing}(\Sigma_k))=0$, and having a uniform mass ratio upper bound \(C\). Assume that \(c_{k}\to c\in [0,+\infty)\). Then there  exists an integral varifold $V$, such that (up to subsequences) $\Sigma_k$ converge in the sense of varifold to $V$ and
  \[
    \spt\|V\|\cap B_{r}=\Sigma\cap B_{r},
    \]
  where $\Sigma$ is a a \(c\)-stable hypersurface   uniform mass ratio upper bound \(C\) and $\mathcal H^{n-2}(\Sing(\Sigma)\cap B_{r})=0$. Moreover, for any open set \(U\subset B_{r}\setminus \Sing(\Sigma)\), the converges of \(\Sigma_{k}\) to \(\Sigma\) is locally smooth.
\end{theorem}

The proof is basically contained in~\cite{Allard}, where  no singular set is a priori allowed. It is however classical (see also the end of the Introduction~\cite{Allard}) to combine the mass ratio  bound with a capacity argument to show that one can deal with a singular set of vanishing \(\mathcal H^{n-2}\) measure since the singular set  is not seen by the stability inequality. In particular,  while a-priori the stability inequality is  only valid for \(\varphi \in C_c^1\) whose support does not intersect \(\Sing (\Sigma)\), a capacity argument identical to the one at the end of Section 2 in~\cite{SchoenSimon} implies that it can be extended to all \(\varphi\). In particular we record for future use the following generalization of~\eqref{eq:stability}
\begin{equation}
    \label{useful1}
        \int_\Sigma \varphi^2|A_\Sigma|^2\le C\int_\Sigma |\nabla \varphi|^2+\varphi^2. \qquad \text{for all \(\varphi\in C_c^1(\Sigma)\)}
\end{equation}
 With this caveat in mind we state the following, which is~\cite[Section 4.1, The Main Theorem]{Allard}.

\begin{theorem}\label{reg}
Let \(\Phi\) be an integrand satisfying~\eqref{limitato}  and \(C>0\). Then there exists a constant \(\bar\varepsilon=\varepsilon(n,\lambda,C)\) such that for  $c$-stable almost embedded hypersurface $\Sigma \subset M$ having a mass ratio upper bound \(C\) and  satisfying $\mathcal{H}^{n-2}(\operatorname{Sing}(\Sigma))=0$ the following holds.
If $R\in (0,\infty)$, $0\in \Sigma$,
\[
cR+R^{-n}\int_{\Sigma \cap C_{3R}} 1-\langle \nu(x), e_{n+1} \rangle^2 d\mathcal H^n \leq \bar\varepsilon.
\]
Then there exists a positive $\kappa \in \mathbb N$ and $C^2$ functions $f_i:D_{R} \to \mathbb R$, $i=1, \dots, \kappa$, such that
  \[
  \Sigma \cap C_{R}=\cup_{i=1}^\kappa \operatorname{graph}(f_i) \cap C_{R}
  \]
  and
  \[
  \|Df_i\|_\infty\leq C(n,\lambda)\bar\varepsilon \qquad \forall i=1,\dots,\kappa
  \]
  Here we have denoted by \(D_{R}\) the \(n\) dimensional disk  of radius \(R\) and center \(0\) and by \(C_{R}=D_{R}\times \mathbb R\) the \((n+1)\) dimensional cylinder.
\end{theorem}

We also need the following classical lemma:
\begin{lemma}\label{lemmaintermedio}
  Consider a sequence of $c_k$-stable almost embedded hypersurfaces $\Sigma_k \subset M$ with respect to uniformly elliptic anisotropic integrands $\Phi_k$, with singular set satisfying $\mathcal{H}^{n-2}(\operatorname{Sing}(\Sigma_k))=0$, and having a uniform mass ratio upper bound $C_0$, such that
  \begin{itemize}
  \item[(i)] $\sup_k \mathcal H^n(\Sigma_k \cap C_{3R})<\infty$;
  \item[(ii)] $\sup_k \|\Phi_k\|_{C^3}<\infty$ and $\lim_k\|D_x\Phi_k\|_\infty = 0$;
  \item[(iii)] $\lim_k c_k=0$;
  \item[(iv)] $\lim_k\int_{\Sigma_k\cap C_{3R}}|\langle x, e_{n+1} \rangle |d\mathcal H^n=0$.
  \end{itemize}
  Up to subsequences we have that
  \[
    \lim_k \int_{\Sigma_k \cap C_{3R}} 1-\langle \nu_k(x), e_{n+1} \rangle^2 d\mathcal H^n=0.
    \]
\end{lemma}
\begin{proof}
  To this aim, we just need to show that the varifolds $V_k$ canonically associated to $\Sigma_k$ converge (up to subsequences) in $C_{3R}$ to a varifold $V$ satisfying
  \begin{equation}\label{obiettivo}
    \int_{\Sigma_k \cap C_{3R}} 1-\langle \nu, e_{n+1} \rangle^2 dV(x,\nu)=0.
  \end{equation}
  This is easily verified because by (i) (up to subsequences) $V_k\rightharpoonup V$ in $C_{3R}$. By (ii) we deduce that
  \[
    \delta_{\mathbf{\Phi}_k}V_k\rightharpoonup \delta_{\mathbf{\Phi}_0}V \quad \mbox{in } C_{3R}
    \]
  where $\Phi_0 = \lim_k \Phi_k$ and $D_x\Phi_0=0$. Hence, by (iii) we deduce that $\delta_{\mathbf{\Phi}_0}V=0$ in $C_{3R}$.
  We test $\delta_{\Phi_0}V$  with  the vector field $g(x):=\varphi(x)\langle x, e_{n+1}\rangle D\Phi_{0}(e_{n+1})$, where $\varphi$ is a smooth compactly supported function in $C_{3R}$  obtaining  thanks to~\eqref{e:firstvariation}
\begin{equation}\label{zero}
\begin{split}
      0&=\delta_{\mathbf{\Phi_0}}V(g)=\int \langle \Phi_0(\nu)\Id-\nu\otimes D\Phi_0(\nu), Dg \rangle dV(x,\nu)
      \\
      &=\int (\Phi_0(\nu)\langle D\varphi(x),D\Phi(e_{n+1})\rangle-\langle \nu, D\Phi_{0}(e_{n+1}\rangle\langle D\varphi(x), D\Phi_0(\nu)\rangle)\langle x, e_{n+1}\rangle dV(x,\nu)
\\
           &\quad+ \int \varphi(x)(\Phi_0(v)\langle D\Phi_{0}(e_{n+1}), e_{n+1}\rangle-\langle \nu, D\Phi_{0}(e_{n+1}\rangle\langle e_{n+1}, D\Phi_0(\nu)\rangle) dV(x,\nu)
          \\
      &=\int \varphi(x)\Bigl(\Phi_0(v)\Phi_{0}(e_{n+1})-\langle \nu, D\Phi_{0}(e_{n+1}\rangle\langle e_{n+1}, D\Phi_0(\nu)\rangle\Bigr) dV(x,\nu).
\end{split}
\end{equation}
  Here we used that, by (iv), \(\spt V\subset\{\langle x, e_{n+1}\rangle\}=0\) and that, by homogeneity,
   \[
   \Phi_{0}(e_{n+1}) =\langle D\Phi_{0}(e_{n+1}), e_{n+1}\rangle.
   \]
   By the strict convexity of \(\Phi_{0}\),
   \[
   \Phi_0(v)\Phi_{0}(e_{n+1})-\langle \nu, D\Phi_{0}(e_{n+1}\rangle\langle e_{n+1}, D\Phi_0(\nu)\rangle>0
   \]
   unless \(\nu=\pm e_{n+1}\), see for instance~\cite[Proof of Theorem 1.3]{DPDRGrect}. This implies~\eqref{obiettivo} and concludes the proof.
    \end{proof}

\begin{proof}[Proof of Theorem~\ref{compactness}]
  The convergence (up to subsequence) in the sense of varifold of $\Sigma_k$ to $V$ follows by weak compactness together with the uniform mass ratio upper bound. Moreover $V$ satisfies the following bound on the anisotropic first variation:
  \[
    \|\delta_\mathbf{\Phi} V\|\leq c \|V\|.
    \]
  We also observe that, since $\mathcal{H}^{n-2}(\operatorname{Sing}(\Sigma_k))=0$, with the same capacity argument of~\cite{SchoenSimon}, one can obtain the  lower density estimate in~\cite[Section 2.2, Theorem]{Allard} for $\Sigma_k$ and conclude that $V$ has positive density for $\|V\|$-a.e.\ point in $B_{4r}$. Hence, by the compactness result~\cite[Theorem 4.1]{DeRosa}, we conclude that $V$ is an integral varifold.

We now aim to show that, with the exception of an \((n-2)\)-null set, \(V\) is supported on a smooth hypersurface \(\Sigma\) and that \(\Sigma_{k}\) converge
locally smoothly to \(\Sigma\).  To this end we consider the measures:
\[
 \mu_k:=|A_{\Reg(\Sigma_k)}|^{2} \mathcal{H}^n \res \Reg(\Sigma_k) \cap B_{2r}
\]
where we have set \(\Reg(\Sigma_{k})=\Sigma_{k}\setminus \Sing(\Sigma_{k})\). By~\eqref{useful1} and the uniform mass ratio upper bound we deduce that
\[
\sup_{k}\mu_{k}(B_{2r})<+\infty.
\]
Hence, up to a subsequence, \(\mu_{k}\) weakly*  converge to a measure \(\mu\). We define
\begin{equation*}\label{s}
\mathcal S=\Bigl\{ x\in \spt V: \limsup_{t \to 0}\frac{\mu(B_{t}(x))}{t^{n-2}}>0\Bigr\}.
\end{equation*}
 We claim that at  every point  \(x\in \spt V \setminus \mathcal S\) the convergence of \(\Sigma_{k}\) to \(V\) is locally smooth, in particular \(\spt V\) is regular in a neighborhood of \(x\) and
 \begin{equation}\label{singset}
 \Sing(\Sigma)\subset \mathcal S.
 \end{equation}
Fix any such \(x\) and let \(W\in T(x,V)\) be a tangent varifold to \(V\) along a sequence \(r_{j}\). Note that
\[
\lim_{j}\lim_{k}(\eta_{x,r_{k}})_{\sharp}|\Sigma_{k}|=W \qquad \text{and} \qquad \lim_{j} \lim_{k} \frac{\mu_{k}(B_{r_{j}}(x))}{r_{j}^{n-2}}=0.
\]
Hence, by a standard diagonal argument, we get the existence of sequence \(r_{j}\) such that
\[
V_{j}:=(\eta_{x,r_{j}})_{\sharp}|\Sigma_{j}|\to W \qquad \text{and} \qquad \lim_{j}\frac{\mu_{j}(B_{r_{j}}(x))}{r_{j}^{n-2}}=0.
\]
Denoting by \(\delta W\)  the first variation with respect to the \emph{area functional} we get that
\[
\begin{split}
|\delta W|(B_{R})&\le
\liminf_{j} |\delta V_{j}|(B_{1})
=\liminf_{j}  r_{j}^{1-n}|\delta |\Sigma_{j}||(B_{r_{j}}(x))
\\
&\le  \limsup_{j}  r_{j}^{1-n}\int_{\Sigma_{j}\cap B_{r_{j}}(x)}|A_{\Sigma_{j}}| \le  C \limsup_{j} r_{j}^{1-n/2}\sqrt{\mu_{j}(B_{r_{j}}(x))}=0,
\end{split}
\]
where in the last inequality we have used H\"older inequality and the mass ratio upper bound.
We now consider \(W' \in T(0,W)\). By the classical monotonicity formula (recall that \(W\) is \emph{area-stationary}) we get that \(W'\) is a cone. Furthermore, by a diagonalization argument as the one above, we get that it can be obtained as limit of a suitable rescaling  of \(\Sigma_{k}\) with asymptotically vanishing second fundamental form. This implies that the regular part of \(W'\) is contained in the union of countably  many hyperplanes passing through the origin.\footnote{This follows from \(x\in T_{x}W'\) since \(W'\) is a cone}
We now claim that  \(\spt W'\) consists of a unique hyperplane. Note that this implies~\eqref{singset}, since this would imply that there is a hyperplane among the tangent varifolds to \(V\) at \(x\) and  thus the assumptions of  Theorem~\ref{reg}  are  satisfied for \(\Sigma_{k}\) a suitable small ball around \(x\).

To prove that the support consists of a single hyperplane we note that, by  the constancy lemma~\cite{Allard1972},  the regular part of \(W'\) coincides with the union of the hyperplanes  minus their intersection.  Monotonicity formula  and integrality implies  the hyperplanes are indeed finitely many. Let us take a point \(y\in \spt W\) such that all the hyperplanes passing through \(y\) intersects along a common subspace of dimension \(n-1\). Note that such a point necessarily exists (but it might not be the origin). If we take a further blow up at \(y\) we get that the tangent cone \(W''\) to \(W'\) at \(y\) is, after a rotation,  of the form
\[
W''=W_{1}\times \mathbb R^{n-1}
\]
where \(W_{1}\subset \mathbb R^{2}\) is a union of half-lines through the origin. Moreover, by a diagonal argument, \(W''\) can be obtained as limit of suitable rescaling and translation of the \(\Sigma_{k}\). The fact that \(W_{1}\) then consists of a single line can be proved by the same argument as in~\cite[Page 786-787]{SchoenSimon}, which is based on the \(L^{2}\) curvature bound, which follows from~\ref{useful1} and the mass ration upper-bound. As explained above, this concludes the proof of~\eqref{singset}.

We are now left to show that \(\mathcal S\) has  vanishing \(\mathcal H^{n-2}\) measure. This is again based on a capacity argument.  It would indeed be clearly enough to show it for
\[
\mathcal S_{\delta}=\Bigl\{ x\in \spt V: \limsup_{t \to 0}\frac{\mu(B_{t}(x))}{t^{n-2}}>\delta\Bigr\},
\]
for every \(\delta >0\). By a simple covering argument, compare with~\cite[Theorem 6.9]{Mattila}, one gets that for every \(K\subset \mathcal S_{\delta}\) and every open set containing \(K\):
\begin{equation}\label{eUK}
\mathcal H^{n-2}(K) \le C(n,\delta) \mu (U).
\end{equation}
In particular \(\mathcal H^{n-2}(K)\) is finite.
By arguing as in~\cite[Theorem 4.16]{EvansGariepy}, we can use this fact and the mass ratio bound of \(\Sigma_{k}\) to deduce that there exits a sequence of open neighborhood of \(V_{j}\) of \(K\) and of functions \(g_{j}\in C^{1}_{c}\) such that \(V_{j}\subset \{g_{j}\ge 1\) and
\[
\sup_{k}\int_{\Sigma_{k}} |\nabla g_{j}|^{2}+g^{2}_{j} =o_{j}(1),
\]
where \(o_{j}(1)\to 0\) as \(j\to \infty\). In particular, by~\eqref{useful1},
\[
\begin{split}
\mu(V_{j})&\le \liminf_{k} \mu_{k} (V_{j})\le  \liminf_{k} \int_{\Reg \Sigma_{k}}g_{j}^{2}|A_{\Reg \Sigma_{k}}|^{2}
\\
&\le C\sup_{k} \int_{\Sigma_{k}} |\nabla g_{j}|^{2}+g^2_{j}=o_{j}(1).
\end{split}
\]
By combining this with~\eqref{eUK}, we get that \(\mathcal H^{n-2}(K)=0\) for all \(K\subset \mathcal S_{\delta}\) and thus, by~\cite[Theorem 8.19]{Mattila} also \(\mathcal H^{n-2}(\mathcal S_{\delta})=0\), concluding the proof.
\end{proof}

\section{Proof of Lemma~\ref{lem:trig_number}}\label{sec:geometric}

\begin{proof}[Proof of Lemma~\ref{lem:trig_number}]
  Let $N_1$ be the number of simplexes in $T_0$. For each $\sigma \in T_0$, let $g_0$ be the Euclidean metric for which the \(g\)-length of the side of \(\sigma\) equals its \(g_{0}\)-length. We may assume that $\sigma$ is embedded in $\mathbb{R}^{n+1}$. Since $\varepsilon < 1$, $B_r(p) \cap \sigma$ has diameter at most $4r$ and thus, can be covered by a Euclidean ball $\mathbb{B}_{4r}(x)$ of radius $4r$ for some $x \in \mathbb{R}^{n+1}$. By definition, there are exactly $2^{(n+1)k}$ simplexes of $T_k$ in $\sigma$, each of which has Euclidean volume $c_{n} \left(2^{-k}\mu\right)^{n+1}$. Moreover, each simplex of $T_k$ in $\sigma$ that intersects $B_r(p)$ is contained in $\mathbb{B}_{8r}(x)$. Hence, the number of such simplexes is bounded above by
  \[
    \frac{\alpha_{n}(8r)^{n+1}}{c_{n} \left(2^{-k}\mu\right)^{n+1}} \leq \frac{\alpha_{n}\ 16^{n+1}}{c_{n}}\,,
  \]
  where $\alpha_{n}$ is the volume of a unit ball in $\mathbb{R}^{n+1}$. Therefore, we  set
  \[
  C \coloneqq N_1  \cdot \frac{\alpha_{n+1}\ 16^{n+1}}{c_{n+1}},
  \]
  and $B_r(p)$ intersects with at most $C$ simplexes in $T_k$.
\end{proof}

\printbibliography
\end{document}